\documentclass{amsart}
\usepackage[inactive]{srcltx} 
\usepackage{graphicx}                   
\usepackage{amsmath, amsthm, amscd, amsfonts, amssymb}
\usepackage{float}
\newtheorem{thm}{Theorem}[section]

\theoremstyle{definition}

\newtheorem{rem}[thm]{Remark}

\theoremstyle{remark}
\numberwithin{equation}{section}
   
\newcommand{\PP}{\mathcal{P}}

\newcommand{\C}{\mathcal{C}}
\newcommand{\W}{\mathcal{W}}

\newcommand{\V}{\mathcal{V}}
\newcommand{\IR}{\mathbb{R}}
\newcommand{\IP}{\mathbb{P}}
\newcommand{\IN}{\mathbb N}
\newcommand{\R}{\mathcal{R}}

\newcommand{\qq}{\mathrm{q}}
\newcommand{\rr}{\mathrm{r}}
\newcommand{\dd}{\mathrm{d}}
\newcommand{\vp}{\varphi}

\newcommand{\vertiii}[1]{{\left\vert\kern-0.25ex\left\vert\kern-0.25ex\left\vert #1 
    \right\vert\kern-0.25ex\right\vert\kern-0.25ex\right\vert}}

\DeclareMathOperator{\dom}{dom}
\begin{document}
\nocite{*}
\title[DG for the wave equation: a priori error analysis]
{Discontinuous Galerkin for the wave equation: 
a simplified a priori error analysis}
\author[N.~Rezaei]{Neda Rezaei}

\address{Department of Mathematics, 
University of Kurdistan, P. O. Box 416, 
Sanandaj, Iran}

\email{neda.rezaei93@yahoo.com;  n.rezaei@sci.uok.ac.ir}

\author[F.~Saedpanah]{Fardin Saedpanah}

\address{Department of Mathematics, 
University of Kurdistan, P. O. Box 416, 
Sanandaj, Iran, and 
Department of Engineering, University of Borås, 
SE-501 90 Borås, Sweden} 

\email{f.saedpanah@uok.ac.ir;  
fardin.saedpanah@hb.se}

\thanks{{\scriptsize
\hskip -0.4 true cm \emph{MSC(2010)}:  65M60, 35L05.
}}

\begin{abstract}
Standard discontinuous Galerkin methods, based on piecewise 
polynomials of degree $ \qq=0,1$, are considered for 
temporal semi-discretization for second order hyperbolic 
equations. 
The main goal of this paper is to present a simple 
and straightforward a priori error analysis of optimal order 
with minimal regularity requirement on the solution. 
Uniform norm in time error estimates are also proved. 
To this end, energy identities and stability estimates of 
the discrete problem are proved for a slightly more general 
problem. 
These are used to prove optimal order a priori error estimates 
with minimal regularity requirement on the solution. 
The combination with the classic continuous Galerkin finite 
element discretization in space variable is used, 
to formulate a full-discrete scheme. 
The a priori error analysis is presented. 
Numerical experiments are performed to verify the theoretical 
results. 
\end{abstract}

\date{\today}
\maketitle
\section{Introduction}
We study a priori error analysis of the discontinuous Galerkin 
methods of order $\qq=0,1$, 
dG($\qq$), for temporal semi-discretization of the second order 
hyperbolic problems 
\begin{equation}\label{Model} 
  \ddot{u}+A u =f,  \quad t\in (0,T), \quad 
  \textrm{ with }  
   u(0)=u_{0},\ \dot{u}(0)=v_{0},
\end{equation}
where $A$ is a self-adjoint, positive definite, uniformly elliptic 
second-order operator on a Hilbert space $H$. 
We then combine the dG($\qq$) method with a standard continuous 
Galerkin of order $\rr\geq 1$, cG($\rr$), for spatial discretization  
to formulate a full discrete scheme, to be called dG($\qq$)-cG($\rr$). 

We may consider, as a prototype equation for such second order 
hyperbolic equations, $A=-\Delta$ with homogeneous 
Dirichlet boundary conditions. 
That is, the classical wave equation, 
\begin{equation}\label{Wave-eq} 
\begin{array}{lll} 
  \ddot{u}(x,t)-\Delta u(x,t)  =f(x,t) &  
   \textrm{in} \quad \Omega \times  \left( 0,T\right),\\
   u(x,t)=0 &   \textrm{on} \quad \Gamma \times \left( 0,T\right),\\
   u(x,0)=u_{0}(x),\quad \dot{u}(x,0)=v_{0}(x)&  
    \textrm{in} \quad\Omega,
\end{array}
\end{equation}
where  $\Omega $ is a bounded and convex polygonal domain in 
$\IR^{d}$, $d\in\lbrace1,2,3\rbrace$, with boundary $\Gamma$. 
We denote $\dot{u}=\frac{\partial u}{\partial t}$ and 
$ \ddot{u}=\frac{\partial^{2}u}{\partial t^{2}}$. 
The present work applies also to wave phenomena with vector valued 
solution $u:\Omega \times (0,T) \to \IR^d$, such as wave elasticity. 

We may also consider more general equations 
\begin{equation}\label{Model-tilde} 
  \ddot{u}+ \tilde A u =f,  \quad t\in (0,T), \quad 
  \textrm{ with }  
   u(0)=u_{0},\ \dot{u}(0)=v_{0},
\end{equation}
where $\tilde{A}=-\nabla \cdot \kappa\nabla$.  
That is,  
\begin{equation}\label{Wave-eq-tildeA} 
\begin{array}{lll} 
  \ddot{u}(x,t)-\nabla \cdot \kappa\nabla u(x,t)  =f(x,t) &  
   \textrm{in} \quad \Omega \times  \left( 0,T\right),\\
   u(x,t)=0 &   \textrm{on} \quad \Gamma \times \left( 0,T\right),\\
   u(x,0)=u_{0}(x),\quad \dot{u}(x,0)=v_{0}(x)&  
    \textrm{in} \quad\Omega. 
\end{array}
\end{equation}
Here $ \kappa(x) $ is a smooth function and for two positive constants $\kappa_{\min} $ and
 $\kappa_{\max}$,
 \begin{equation*}
    \kappa_{\min}\leq \kappa(x) \leq\kappa_{\max}, \qquad x\in\Omega.
 \end{equation*}
We note that $\kappa(x)$ can also be a uniformly symmetric positive 
definite matrix. 

Throughout this paper, for simplicity, we consider 
\eqref{Wave-eq}, and we remark how the approach is applied 
to \eqref{Wave-eq-tildeA}, too. 
The results and the corresponding proofs for $\tilde A$ are very similar to the case $A$, and therefore we will omit the proofs. 

The discontinuous Galerkin type methods for time or space 
discretization have been studied extensively in the literature 
for ordinary differential equations and parabolic/hyperbolic 
partial differential equations; see, for example, 
\cite{AdjeridTemimi2011, 
BanjaiGeorgoulisLijoka, Castillo, Celiker, Delfour, 
Eriksson2, Eriksson1, Johnson, LarssonRachevaSaedpanah, 
LarssonThomeeWahlbin, Lasaint, 
Kassem, Reed, SchmutzWihler, Thomee} and the references therein. 
In particular, several discontinuous and continuous Galerkin 
finite element methods, both in time and space 
variables, for solving second order hyperbolic equations have 
appeared in the literature, see ,e.g. 
\cite{AdjeridTemimi2011, French, Grote, Johnson1, Riviere} 
and the references therein. 

A dG(1)-cG(1) methods was studied in \cite{Johnson1}. 
This was extended by \cite{AdjeridTemimi2011}, where 
dG time-stepping methods was applied directly to the 
second-order ode system, that arise 
from spatial semi-discretization by standard cG methods. 
Discontinuous spatial discretization of  
wave problems were studied in 
\cite{Grote, MullerSchotzauSchwab2018, Riviere}. 

Uniform in time stability analysis, also so-called 
\textit{strong stability} or 
\textit{$L_\infty$-stability}, 
has been studied for parabolic problems,  
\cite{ErikssonJohnsonLarsson, LarssonThomeeWahlbin, 
SchmutzWihler}, but not for second order hyperbolic problems. 
An important tool for such analysis for parabolic problems is 
the smoothing property of the solution operator, 
thanks to analytic semigroup. 
For parabolic problems, in \cite{ErikssonJohnsonLarsson}, 
uniform in time stability 
and error estimates for dG($\qq$), $\qq\geq 0$, have been proved using 
Dunford-Taylor formula based on smoothing properties of the 
analytic semigroups. For parabolic problems which is perturbed 
by a  memory term, such analysis has been done for dG(0) and 
dG(1), using the linearity of the basis functions in time,  
\cite{LarssonThomeeWahlbin}. 
Another way to analyze uniform in time stability is 
using a lifting operator technique to write the dG($\qq$) 
formulation in a strong (pointwise) form, \cite{SchmutzWihler}. 

Second order hyperbolic problems unfortunately do not enjoy such 
smoothing properties, due to the fact that the solution operator 
generates a $C_0$-semigroup only, but not analytic semigroup. 
However, one can use linearity of the basis function in time 
in case of dG(0) and dG(1) to prove such a priori error estimates, 
that is a part of this work. 

Optimal order $L_\infty([0,\infty),L_2(\Omega))$ estimates for 
Galerkin finite element approximation of 
the wave equation were first obtained by \cite{Dupont}, and 
the regularity requirement for the initial displacement was 
not minimal. 
This was improved in \cite{Baker}, and in \cite{Rauch} 
it was shown that the resulting regularity requirement 
is optimal, see \cite[Lemma 4.4]{KovacsLarssonSaedpanah2010} 
for more details. 
A new approach was introduced for a priori error analysis of 
the second order hyperbolic problems in the context of 
continuous Galerkin methods, 
spatial semi-discretization cG(1) in 
\cite{KovacsLarssonSaedpanah2010} and cG(1)-cG(1) in 
\cite{LarssonSaedpanah2010}. 

Here, we extend such a priori error analysis to dG($\qq$) 
time-stepping for $\qq=0,1$, 
for \eqref{Wave-eq}, as the chief example for \eqref{Model}. 
We also present the a priori error analysis for a full discrete 
scheme by combining dG($\qq$) with a standard cG($\rr$), $\rr\geq 1$, 
method for spatial discretization (see also Remark \ref{Remark_1}). 
The regularity requirements on the solution is minimal, that 
is important, in particular, for stochastic model problems and 
for second order hyperbolic partial differential equations 
perturbed by a memory term, see  
\cite{KovacsLarssonSaedpanah2010, LarssonSaedpanah2010, 
Saedpanah2015}. 
The approach presented here is simple and straightforward 
such that we can prove error estimates in several space-time 
norms. We show also how the same approach is used to prove 
uniform in time error estimates. 
We note that the error analysis in 
\cite{LarssonSaedpanah2010} is based on energy arguments, while 
in \cite{Saedpanah2015} it is via duality arguments. 
That is, we can use the presented approach of error analysis 
of dG methods via duality arguments, too.

To prove a priori error estimates 
\textit{at the time-mesh points} and also 
\textit{uniform in time}, we prove stability estimates and 
energy identity, respectively, 
for the discrete problem of a more general form, by considering 
an extra (artificial) load term in the so called 
displacement-velocity formulation (see Remark \ref{ExtraLoad}). 
This gives the flexibility to obtain optimal order a priori 
error estimates with minimal regularity requirement on the 
solution. See Remark \ref{Remark_4}, too. 
For dG methods long-time integration without error accumulation 
is possible, since the stability constants are independent of 
the length of the time interval, see also Remark 
\ref{Remark_dG_cG}.

The outline of this paper is as follows.
We provide some preliminaries and the weak formulation of the 
model problem, in $\S 2$. 
In section 3, we formulate the dG($\qq$) method,  
and we obtain energy identity and stability estimates   
for the discrete problem of a slightly more general form. 
Then, in $\S 4$, we prove optimal order a priori error estimates 
in $L_2$ and $H^1$ norms for the displacement and $L_2$-norm 
of the velocity, with minimal regularity requirement on the 
solution. 
We also prove uniform in time a priori error estimates.  
In $\S$ \ref{FullDiscrete}, we formulate 
the dG($\qq$)-cG($\rr$) scheme and study the stability of the discrete 
problem, to be used to prove a priori error estimates in section 
\ref{A priori_Full}.  
Finally, numerical experiments are presented in section 
\ref{Example} in order to illustrate the theory.

\section{Preliminaries}
We let $ H=L_{2}( \Omega)$ with the inner product 
$(\cdot,\cdot)$ and the induced norm $\| \cdot \|$. 
Denote 
$\V=H_0^1(\Omega)=\{ u\in H^1(\Omega): u|_{\Gamma}=0 \}$ 
with the energy inner product 
$a(\cdot, \cdot)=(\nabla\cdot,\nabla\cdot)$ and  the induced norm 
$\| \cdot \|_{\V}$.
Let $A=-\Delta$ be defined with homogeneous Dirichlet boundary conditions on $ \dom\left( A\right)=H^{2}(\Omega)\cap\V$, 
and 
$\{(\lambda_k,\vp_k)\}_{k=1}^\infty$ be the eigenpairs of $A$, i.e.,
\begin{equation*} 
 A\vp_k=\lambda_k \vp_k,\quad k \in \IN.
\end{equation*}
It is known that
$ 0<\lambda_1\leq \lambda_2\leq \dots \leq \lambda_k \leq \cdots$
with $\lim_{k\to\infty}\lambda_k=\infty$ and the eigenvectors
$\{\vp_k\}_{k=1}^\infty$ form an orthonormal basis for $H$.  
Then 
\begin{equation*}
  (A^{l}u,v) =\sum_{k=1}^{\infty} \lambda^{l}_{k}(u,\varphi_{k})(v,\varphi_{k}), 
\end{equation*}
and we introduce the fractional order spaces, \cite{Thomee}, 
\begin{equation*}
 \dot H^\alpha
 =\dom(A^{\frac{\alpha}{2}}), \quad
 \|v\|_\alpha^2 :=\|A^{\frac{\alpha}{2}} v\|^2
 =\sum_{k=1}^\infty \lambda_k^\alpha (v,\vp_k)^2,\quad 
 \alpha\in \IR,\ v \in \dot H^\alpha.
\end{equation*}
We note that $H=\dot H^0$ and $\V = \dot H^1$.
Defining the new variables $  u_{1}=u $ and $ u_{2}=\dot{u} $,
we can write the velocity-displacement form of  \eqref{Wave-eq} as
\begin{equation*} 
\begin{array}{lll} 
   - \Delta\dot{u}_{1} +\Delta u_{2}=0
     &\textrm{in} \quad \Omega \times  \left( 0,T\right),\\
     \dot{u}_{2}-\Delta u_{1}=f    
     &\textrm{in} \quad  \Omega \times \left( 0,T\right),\\
     u_{1}=u_{2}=0  
     &\textrm{on}  \quad \Gamma \times \left( 0,T\right),\\
     u_{1}(\cdot,0)=u_{0}, \ u_{2}(\cdot,0)=v_{0}
     &\textrm{in}  \quad\Omega,
\end{array} 
\end{equation*}
for which, the weak form is to find
$ u_{1}(t) $ and $ u_{2}(t) \in \V $ such that  
\begin{equation}\label{weak form}
\begin{array}{l} 
   a(\dot{u}_{1}(t),v_{1})-a(u_{2}(t),v_{1})=0,\\
   (\dot{u}_{2}(t),v_{2})+ a(u_{1}(t),v_{2})=( f(t),v_{2}),
   \qquad  \forall v_{1}, v_{2} \in \V,\quad t\in(0,T),\\
   u_{1}(0)=u_{0}, \ u_{2}(0)=v_{0}. 
\end{array}
\end{equation}
This equation is used for dG($\qq$) formulation.

\begin{rem} \label{A tilde: section 2}
For $\tilde{A}=-\nabla \cdot \kappa\nabla$ 
in \eqref{Wave-eq-tildeA}, 
with homogeneous Dirichlet boundary conditions on 
$ \dom( \tilde{A})=H^{2}(\Omega)\cap\V$, 
we denote by $\{(\tilde{\lambda}_k,\tilde \vp_k)\}_{k=1}^\infty$ 
 the eigenpairs of $\tilde{A}$, i.e.,
\begin{equation*} 
 \tilde{A}\tilde \vp_k=\tilde{\lambda}_k \tilde\vp_k,\quad k \in \IN.
\end{equation*}
Then 
$ 0<\tilde\lambda_1\leq \tilde\lambda_2\leq \dots \leq \tilde\lambda_k \leq \cdots$
with $\lim_{k\to\infty}\tilde\lambda_k=\infty$ and the eigenvectors
$\{\tilde\vp_k\}_{k=1}^\infty$ form an orthonormal basis for $H$. 
Having 
\begin{equation*}
  (\tilde A^{l}u,v) 
  =\sum_{k=1}^{\infty} \tilde\lambda^{l}_{k}
    (u,\tilde\varphi_{k})(v,\tilde\varphi_{k}), 
\end{equation*}
we can introduce the fractional order spaces 
\begin{equation*}
 \dot{ \tilde H}^\alpha
 =\dom(\tilde A^{\frac{\alpha}{2}}), \quad
 \vertiii{v}_\alpha^2 :=\|\tilde A^{\frac{\alpha}{2}} v\|^2
 =\sum_{k=1}^\infty \tilde \lambda_k^\alpha (v,\tilde \vp_k)^2,\quad 
 \alpha\in \IR,\ v \in \dot{ \tilde H}^\alpha.
\end{equation*}
We note that $H=\dot{\tilde{H}}^0$ and $\V=\dot{\tilde{H}}^1$. 
If we define an energy inner product 
$ \tilde{a}(\cdot,\cdot)=(\kappa \nabla\cdot,\nabla\cdot) $ 
with the induced norm $\vertiii{ \cdot}_{\V}$, then the norms 
$ \vertiii{ \cdot}_{\V} $ and $ \| \cdot \|_{\V} $  
are equivalent on $ \V $, that is 
\begin{equation*} 
 \kappa_{\min}\| v \|_{\V}
 \leq\vertiii{ v}_{\V}
 \leq \kappa_{\max}\| v \|_{\V}, \qquad v \in  \V .
\end{equation*}
We also note that 
the norms $\|\cdot \|_\alpha$ and $\vertiii{\cdot}_\alpha$ 
are equivalent on $\dot H^\alpha$. 

Then, the weak form of \eqref{Wave-eq-tildeA} is to find 
$ u_{1}(t) $ and $ u_{2}(t) \in \V $ such that  
\begin{equation}\label{weak form-tilde}
\begin{array}{l} 
   \tilde a(\dot{u}_{1}(t),v_{1})-\tilde a(u_{2}(t),v_{1})=0,\\
   (\dot{u}_{2}(t),v_{2})+ \tilde a(u_{1}(t),v_{2})=( f(t),v_{2}),
   \qquad  \forall v_{1}, v_{2} \in \V,\quad t\in(0,T),\\
   u_{1}(0)=u_{0}, \ u_{2}(0)=v_{0}. 
\end{array}
\end{equation}
This equation then can be used for the dG($\qq$) formulation. 

\end{rem}

\section{The discontinuous Galerkin time discretization} \label{dG}
In this section, we apply the standard dG method in time variable using 
piecewise polynomials of degree $\qq=0,1$, 
and we investigate the stability. 

\subsection{dG($ \qq $) formulation}
Let $ 0=t_{0}<t_{1} <\cdots <t_{N} =T$ be a temporal mesh 
with time subintervals $I_{n}=(t_{n-1},t_{n})$ 
and steps $ k_{n}=t_{n}-t_{n-1} $, 
and the maximum step-size by $  k=\max_{1\leq n\leq N}k_{n}$. 
Let $\IP_{\qq}=\IP_{\qq}(\V)
 =\{  v: v(t)=\sum_{j=0}^{\qq}v_{j}t^{j}, v_{j}\in \V\} $
and define the finite element space 
$ \V_{\qq} =\{ v: v|_{S_n}\in \IP_{\qq}( \V ),\  n=1, \dots, N\} $
for each space-time 'Slab' $ S^{n}=\Omega \times I_{n} $.
 
We follow the usual convention that a function $ U=(U_{1},U_{2})\in \V_{\qq}\times  \V_{\qq}$ 
is left-continuous at each time level $ t_{n} $ and we define 
$  U_{i,n}^{\pm}=\lim_{s\rightarrow 0^{\pm}}U_{i}(t_{n}+s)$, writing
\begin{equation*}
  U_{i,n}^{-}=U_{i}(t_{n}^{-}),
  \quad U_{i,n}^{+}=U_{i}(t_{n}^{+}),
  \quad\left[  U_{i}\right]_{n}=U_{i,n}^{+}-U_{i,n}^{-}
  \quad \textrm{for}\hskip .2cm i=1,2.
\end{equation*}

The dG method determines 
$ U=(U_{1},U_{2})\in \V_{\qq}\times  \V_{\qq} $ on 
$ S^{n} \times S^{n} $
for $ n=1,\dots,N $ by setting 
$ U_{0}^{-}=(U_{1,0}^{-},U_{2,0}^{-})$, and then
\begin{equation}\label{dG-I_n}
\begin{split}
 &\int_{I_{n}}\Big( a(\dot{U}_{1},V_{1})
   -a(U_{2},V_{1})\Big) \dd t
   +a(U_{1,n-1}^{+},V_{1,n-1}^{+})
  = a(U_{1,n-1}^{-},V_{1,n-1}^{+}),\\
 &\int_{I_{n}}\Big((\dot{U}_{2},V_{2})
   +a(U_{1},V_{2})\Big) \dd t 
   +( U_{2,n-1}^{+},V_{2,n-1}^{+})\\
 &\hskip 1.5cm=( U_{2,n-1}^{-},V_{2,n-1}^{+})
   +\int_{I_{n}}(f,V_{2}) \dd t, 
    \quad\forall 
    V=(V_{1},V_{2})\in \IP_{\qq}\times \IP_{\qq} .
\end{split}
\end{equation}

Now, we define the function space $\W$ consists of functions 
which are piecewise smooth with respect to the temporal mesh with values
in $ \dom(A)$. 
We note that $\V_{\qq}\subset \W $.
Then we define the bilinear form 
$B$ 
and the linear form $L$ on $\W \times \W$  
by 
\begin{equation}\label{B(u,v)=L(v)}
\begin{split}
  & \hskip -1cm B( (u_{1},u_{2}),(v_{1},v_{2}))
  =\sum_{n=1}^{N}\int_{I_{n}}\Big\{ a(\dot{u}_{1},v_{1}) 
   -a(u_{2},v_{1})
   +(\dot{u}_{2},v_{2})
   +a(u_{1},v_{2}) \Big\}\dd t \\
  & \hskip 4.2cm + \sum_{n=1}^{N-1}
    \left\lbrace  a(\left[ u_{1}\right]_{n},v_{1,n}^{+} )
    +(\left[ u_{2}\right]_{n},  v_{2,n}^{+})\right\rbrace\\
  &\hskip 4.2cm+a(u_{1,0}^{+},v^{+}_{1,0})
   +( u^{+}_{2,0},v^{+}_{2,0}),\\
  &\hskip 0.2cm L\big((v_{1},v_{2})\big)
 =\sum_{n=1}^{N}\int_{I_{n}}( f,v_{2}) \dd t 
   +a(u_{0},v^{+}_{1,0})+( v_{0},v_{2,0}^{+}).
\end{split}
\end{equation}
Then $ U=(U_{1},U_{2})\in \V_{\qq}\times  \V_{\qq}$, 
the solution of discrete problem \eqref{dG-I_n}, satisfies 
\begin{equation}\label{B(U,V)}
\begin{split}
 &B(U,V)=L(V), \qquad\forall V=(V_{1},V_{2})\in  \V_{\qq}\times  \V_{\qq},\\
 &U^{-}_{0}=(U^{-}_{1,0},U^{-}_{2,0})=(u_{0},v_{0}).
\end{split}
\end{equation}
We note that the solution $u= (u_{1},u_{2})$ of \eqref{weak form} also satisfies 
\begin{equation}\label{B(u,v)}
\begin{split}
 &B (u,v)=L(v),  \qquad \forall  v=  (v_{1},v_{2})\in \W \times \W,\\
 &(u_{1}(0),u_{2}(0))=(u_{0},v_{0}).
\end{split}
\end{equation}
These imply the Galerkin orthogonality for the error
$ e=(e_{1},e_{2}) =(U_{1},U_{2})-(u_{1},u_{2})$, that is, 
\begin{equation} \label{GalerkinOrthogonality}
\hskip -.6cm B(e,V) =0, \quad  \forall  V=(V_{1},V_{2})\in   \V_{\qq}\times  \V_{\qq}.
\end{equation}
Integration by parts yields an alternative expression for the bilinear form \eqref{B(u,v)=L(v)}, as
\begin{equation}\label{B^*}
\begin{split}
 B^*(u,v)
 & =\sum_{n=1}^{N}\int_{I_{n}}\Big\{-a(u_{1},\dot{v}_{1})
   -a(u_{2},v_{1})-( u_{2},\dot{v}_{2})
   +a(u_{1},v_{2})\Big\} \dd t\\
 &\quad -\sum_{n=1}^{N-1}\Big\{a(u^{-}_{1,n},[v_{1}]_{n} )
   +( u^{-}_{2,n},[v_{2}]_{n})\Big\} \\
 &\quad +a(u^{-}_{1,N},v^{-}_{1,N})+(u^{-}_{2,N},v^{-}_{2,N}).
\end{split}
\end{equation}

\begin{rem} \label{Remark_1}
We note that the framework applies also to spatial finite dimensional 
function spaces $\V_{\qq,\rr}\subset \V_{\qq}$, such as, a continuous Galerkin 
finite element method of order $r$ for discretization in space variable. 
One can combine a continuous Galerkin finite element method 
in spatial variable to get a full discrete scheme. 
That is the subject of section \ref{FullDiscrete}. 
\end{rem}

\subsection{Stability}
Here, we present a stability (energy) identity and stability 
estimate, that are used in a priori error analysis. 
In our error analysis we need a stability identity for a 
slightly more general problem, that is $ U=(U_{1},U_{2})\in  \V_{\qq}\times  \V_{\qq} $ such that 
\begin{equation}\label{B= L^V}
 B(U,V)=\hat{L}(V),   \qquad  \forall V=(V_{1},V_{2})\in  \V_{\qq}\times  \V_{\qq},
\end{equation}
where the linear form 
$\hat{L}$ is defined on $\W\times\W$ by
\begin{align*}
 \hat{L}((v_{1},v_{2}))
 =\sum _{n=1}^{N} \int_{I_{n}}\Big\{ a(f_{1},v_{1})
  +( f_{2},v_{2}) \Big\} \dd t
  +a(u_{0},v^{+}_{1,0}) + ( v_{0},v_{2,0}^{+}).
\end{align*}
That is, instead of \eqref{weak form}, we study stability 
of the dG($\qq $) discretization of a more general problem
\begin{equation*} 
\begin{array}{l} 
   a(\dot{u}_{1}(t),v_{1})-a(u_{2}(t),v_{1})=a(f_1(t),v_{1}),\\
   (\dot{u}_{2}(t),v_{2})+ a(u_{1}(t),v_{2})=( f_2(t),v_{2}),
   \qquad  \forall v_{1}, v_{2} \in \V,\quad t\in(0,T),\\
   u_{1}(0)=u_{0}, \ u_{2}(0)=v_{0}. 
\end{array}
\end{equation*}
See Remark \ref{ExtraLoad}. 

We define the following norms
 \begin{align*}
 \| u\|_{I_{n}}=\sup_{t\in I_{n}} \| u(t)\|,  \quad \text{and}
\quad \| u\|_{s,I_{n}}=\sup_{t\in I_{n}} \| u(t)\|_{s},
\end{align*}
and
 \begin{align*}
 \| u\|_{J_{N}}=\sup_{t\in J_{N}} \| u(t)\|,  \quad \text{and}
\quad \| u\|_{s,J_{N}}=\sup_{t\in J_{N}} \| u(t)\|_{s},
\end{align*}
where $ J_{N}=(0,t_{N})$.



\begin{thm}\label{Theorem-stability1}
Let $ U=(U_{1},U_{2}) $ be a solution of \eqref{B= L^V}. 
Then for any $ T>0 $ and $ l\in\IR $, we have the energy identity
\begin{equation}\label{ stability identity} 
\begin{split}
 \| U_{1,N}^{-}\|_{l+1}^{2}
 +&\| U_{2,N}^{-}\|_{l}^{2}
  +\sum _{n=0}^{N-1}\left\lbrace \|[U_{1}]_{n}\|_{l+1}^{2}
  +\|[U_{2}]_{n}\|_{l}^{2} \right\rbrace  \\
 = &\| u_{0}\|_{l+1}^{2}
  +\| v_{0}\|_{l}^{2}+2\int_{0}^{T}
    \Big\{  a(f_{1},A^{l}U_{1})
     +(f_{2},A^{l}U_{2}) \Big\} \dd t .
\end{split}
\end{equation}
Moreover, for some constant $ C> 0$ (independent of $T$), 
we have the stability estimate 
\begin{align}\label{ stability estimate}
 \| U_{1,N}^{-}\|_{l+1}
  +\| U_{2,N}^{-}\|_{l}
 \leq C\Big(  \| u_{0}\|_{l+1}+\| v_{0}\|_{l}  
  +\int_{0}^{T} \left\lbrace \| f_{1}\|_{l+1}
  +\| f_{2}\| _{l}\right\rbrace  \dd t \Big).
\end{align}
\end{thm}  
\begin{proof}
We set $ V_i=A^{l}U_i$ for~ $ i=1,2 $ in \eqref{B= L^V} to obtain 
\begin{equation*}
\begin{split}
  \frac{1}{2}&\sum_{n=1}^{N}\int_{I_{n}}
 \frac{\partial}{\partial t}\|U_{1}\|_{l+1}^{2}\dd t 
  +\frac{1}{2}\sum_{n=1}^{N}\int_{I_{n}}
   \frac{\partial}{\partial t} \|U_{2}\|_{l}^{2}\dd t \\
 & \ \ +\sum_{n=1}^{N-1}\Big\{a([U_{1}]_{n},A^{l}U_{1,n}^{+})
   +([U_{2}]_{n},A^{l}U_{2,n}^{+})\Big\} 
 +a(U_{1,0}^{+},A^{l}U_{1,0}^{+})
    +(U_{2,0}^{+},A^{l}U_{2,0}^{+})\\
  &=\int_{0}^{T}\Big\{ a(f_{1},A^{l}U_{1})
   +(f_{2},A^{l}U_{2})\Big\} \dd t 
   +a(u_{0},A^{l}U^{+}_{1,0})+(v_{0},A^{l}U^{+}_{2,0}).
\end{split}
\end{equation*}
Now writing the first two terms at the left 
side as
\begin{equation*}
\begin{split}
\hskip 1.5cm \frac{1}{2}\sum_{n=1}^{N}&\int_{I_{n}}
  \frac{\partial} {\partial t}\| U_{1}\|_{l+1}^{2}\dd t 
  +\frac{1}{2}\sum_{n=1}^{N}\int_{I_{n}}
   \frac{\partial}{\partial t}\| U_{2}\|_{l}^{2}\dd t \\
  &=\sum_{n=1}^{N-1} \Big\{\frac{1}{2}\| U^{-}_{1,n}\|_{l+1}^{2}
  -\frac{1}{2}\| U^{+}_{1,n}\|_{l+1}^{2}\Big\}+\frac{1}{2}\| U^{-}_{1,N}\|_{l+1}^{2}
  -\frac{1}{2}\| U^{+}_{1,0}\|_{l+1}^{2} \\
  &\qquad+\sum_{n=1}^{N-1}\Big\{\frac{1}{2}\| U^{-}_{2,n}\|_{l}^{2}
   -\frac{1}{2}\| U^{+}_{2,n}\|_{l}^{2}\Big\}+\frac{1}{2}\| U^{-}_{2,N}\|_{l}^{2}
 -\frac{1}{2}\| U^{+}_{2,0}\|_{l}^{2},
\end{split}
\end{equation*}
we have
\begin{equation*}   
\begin{split}
  &\sum_{n=1}^{N-1}
  \Big\{\frac{1}{2}\| U^{-}_{1,n}\|_{l+1}^{2}
   -\frac{1}{2}\| U^{+}_{1,n}\|_{l+1}^{2}
   +a([U_{1}]_{n},A^{l}U_{1,n}^{+})\Big\}
   +\frac{1}{2}\| U^{-}_{1,N}\|_{l+1}^{2}
   +\frac{1}{2}\| U^{+}_{1,0}\|_{l+1}^{2}\\
  &\quad +\sum_{n=1}^{N-1}
   \Big\{\frac{1}{2}\|U^{-}_{2,n}\|_{l}^{2} 
   -\frac{1}{2}\| U^{+}_{2,n}\|_{l}^{2}
   +([U_{2}]_{n},A^{l}U_{2,n}^{+})\Big\}
   +\frac{1}{2}\| U^{-}_{2,N}\|_{l}^{2}
   +\frac{1}{2}\| U^{+}_{2,0}\|_{l}^{2}\\
  &=\sum_{n=1}^{N}\int_{I_{n}}
  \Big\{ a(f_{1},A^{l}U_{1})
   +(f_{2},A^{l}U_{2})\Big\} \dd t 
   +a(U^{-}_{1,0},A^{l}U^{+}_{1,0})
   +(U_{2,0}^{-},A^{l}U_{2,0}^{+}).
\end{split}
\end{equation*}
Then, using (for $n=1,\dots,N-1$)
\begin{equation*} 
\begin{split}
  \frac{1}{2}\| U^{-}_{1,n}\|_{l+1}^{2}
  -\frac{1}{2}\| U^{+}_{1,n}\|_{l+1}^{2}
  + a([U_{1}]_{n},A^{l}U_{1,n}^{+})
 &=\frac{1}{2}\| [U_{1}]_{n}\|^{2}_{l+1}, \\
  \frac{1}{2}\| U^{-}_{2,n}\|_{l}^{2}
  -\frac{1}{2}\| U^{+}_{2,n}\|_{l}^{2}
  + ([U_{2}]_{n},A^{l}U_{2,n}^{+})
 &=\frac{1}{2}\| [U_{2}]_{n}\|^{2}_{l},
\end{split}
\end{equation*}
we conclude 
\begin{equation*}
\begin{split}
  &\frac{1}{2}\sum_{n=1}^{N-1}\| [U_{1} ]_{n} \|_{l+1}^{2}
   +\frac{1}{2}\| U_{1,N}\|^{2}_{l+1}
   +\frac{1}{2}\| U^{+}_{1,0}\|^{2}_{l+1}
   -a(U_{1,0}^{-},A^{l}U_{1,0}^{+}) \\ 
 &\quad+\frac{1}{2}\sum_{n=1}^{N-1}\|[U_{2}]_{n}\|_{l}^{2}
  +\frac{1}{2}\| U_{2,N}\|^{2}_{l}
  +\frac{1}{2}\| U^{+}_{2,0}\|^{2}_{l}
  -(U_{2,0}^{-},A^{l}U_{2,0}^{+})\\
 \quad= &\int_{0}^{T}\Big\{ a(f_{1},A^{l}U_{1})
  +(f_{2},A^{l}U_{2})\Big\} \dd t .
\end{split}
\end{equation*}
Hence, having 
\begin{equation*}
\begin{split}
  \frac{1}{2}\| U^{+}_{1,0}\|_{l+1} ^{2}
  -a(A^{\frac{l}{2}}U_{1,0}^{-},A^{\frac{l}{2}}U_{1,0}^{+})
 &=\frac{1}{2}\|[U_{1}]_{0}\|_{l+1}^{2}
  -\frac{1}{2}\| U^{-}_{1,0}\|_{l+1}^{2},\\
 \frac{1}{2}\| U^{+}_{2,0}\|_{l} ^{2}
  -(A^{\frac{l}{2}}U_{2,0}^{-},A^{\frac{l}{2}}U_{2,0}^{+})
 &=\frac{1}{2}\|[U_{2}]_{0}\|_{l}^{2}
  -\frac{1}{2}\| U^{-}_{2,0}\|_{l}^{2},
\end{split}
\end{equation*}
we conclude the identity
\begin{equation*}
\begin{split}
 \frac{1}{2}\| U_{1,N}^-\|^{2}_{l+1}
 &+\frac{1}{2}\| U_{2,N}^-\|^{2}_{l}
  +\frac{1}{2}\sum_{n=0}^{N-1}\|[U_{1}]_{n}\|_{l+1}^{2}
  +\frac{1}{2}\sum_{n=0}^{N-1}\|[U_{2} ]_{n}\|_{l}^{2} \\
 &=\frac{1}{2}\| u_{0}\|_{l+1}^{2}
  +\frac{1}{2}\| v_{0}\|_{l}^{2}
  +\int_{0}^{T}\Big\{ a(f_{1},A^{l}U_{1})
  +(f_{2},A^{l}U_{2})\Big\} \dd t .
\end{split}
\end{equation*}

Finally, to prove the stability estimate \eqref{ stability estimate}, 
recalling that all terms on the left side of the stability 
identity \eqref{ stability identity} are non-negative, we have
\begin{equation*}
\begin{split}
  \| U_{1,N}^-\|^{2}_{l+1}+\| U_{2,N}^-\|^{2}_{l}
 \leq \| u_{0}\|_{l+1}^{2}+\| v_{0}\|_{l}^{2}     
 +2\sum_{n=1}^{N}\int_{I_{n}}
 \Big\{  a(f_{1},A^{l}U_{1})
  +(f_{2},A^{l}U_{2})\Big\} \dd t.
\end{split}
\end{equation*}
Using Cauchy-Schwarz inequality, we obtain 
\begin{equation*}
\begin{split}
 \| U_{1,N}^-\|^{2}_{l+1}+\| U_{2,N}^-\|^{2}_{l}
 &\leq  \| u_{0}\|_{l+1}^{2}+\| v_{0}\|_{l}^{2}\\
 &\quad +2\sum_{n=1}^{N}\int_{I_{n}}
   \Big\{  \| f_{1}\|_{l+1}\|U_{1}\|_{l+1}
  +\| f_{2}\|_{l}\|U_{2}\|_{l}\Big\} \dd t\\
 & \leq  \| u_{0}\|_{l+1}^{2}+\| v_{0}\|_{l}^{2}
 +2\Big(\|U_1 \|_{l+1,J_N} 
  \sum_{n=1}^{N}\int_{I_{n}} \| f_{1}\|_{l+1}\dd t\Big)\\
 &\quad+2\Big( \|U_2 \|_{l,J_N}
   \sum_{n=1}^{N}\int_{I_{n}} \| f_{2}\|_{l}\dd t\Big), 
\end{split}
\end{equation*}
that, having $2ab \leq \epsilon a^2+ \frac{1}{\epsilon}b^2$, implies
\begin{equation} \label{Stability-Thm 3.2}
\begin{split}
 \| U_{1,N}^-\|^{2}_{l+1}+\| U_{2,N}^-\|^{2}_{l}
 &\leq  \| u_{0}\|_{l+1}^{2}+\| v_{0}\|_{l}^{2}\\
 &\quad +\varepsilon_{1}\|U_1 \|_{l+1,J_N}^{2}
   +\dfrac{1}{\varepsilon_{1}} \Big(\sum_{n=1}^{N}\int_{I_{n}}
    \| f_{1}\|_{l+1}\dd t\Big)^{2}\\
 &\quad+\varepsilon_{2}\|U_2 \|_{l,J_N}^{2}
   +\dfrac{1}{\varepsilon_{2}}\Big( \sum_{n=1}^{N}\int_{I_{n}}
 \| f_{2}\|_{l}\dd t\Big)^{2}.
\end{split}
\end{equation}
Now, using the fact that for piecewise constant and piecewise 
linear functions, i.e., for 
$(U_1,U_2)\in \V_{\qq} \times \V_{\qq},\ \qq=0,1$, 
we have 
\begin{equation*}
  \|U_i\|_{s,J_N} = \max_{n} \|U_i\|_{s,I_n} 
  \leq \max_{n} \|U_{i,n}\|_{s}, \quad 
\text{and} 
\quad  \|U_i\|_{s,J_N}^2 \leq \max_{n} \|U_{i,n}\|_{s}^2, 
\end{equation*}
and that the inequality \eqref{Stability-Thm 3.2} holds for 
arbitrary $N$, we conclude in a standard way 
\begin{equation*}
\begin{split}
  \| U_{1,N}^-\|^{2}_{l+1}+\| U_{2,N}^-\|^{2}_{l}
 &\leq C\Big(\| u_{0}\|_{l+1}^{2}+\| v_{0}\|_{l}^{2} 
   +\big( \int_{0}^{T}\! \| f_{1}\|_{l+1}\dd t\big)^{2}
   +\big( \int_{0}^{T}\! \| f_{2}\|_{l}\dd t\big)^{2}\Big).  
\end{split}
\end{equation*}
This concludes the stability estimate \eqref{ stability estimate},  
and the proof is now complete. 

\end{proof}

\begin{rem} \label{A tilde: section 3}
The dG($\qq$) can be applied to 
\eqref{Wave-eq-tildeA}, using the weak form 
\eqref{weak form-tilde}. 
Then stability identity and estimates, 
similar to \eqref{ stability identity} and 
\eqref{ stability estimate}, are obtained with norms 
$\vertiii{\cdot}_s$, the energy inner product 
$\tilde a(\cdot,\cdot)$ and the operator $\tilde A$,  
instead of $\|\cdot\|_s,\ a(\cdot,\cdot)$ and $A$, respectively. 
\end{rem}

\section{A priori error estimates for temporal discretization} \label{A priori}
For a given function $ u\in \C ([0,T]; \V),$ 
we define the interpolation $ \Pi_{k}u\in \V_{\qq}$ by 
\begin{equation}\label{interpolation operator}
\begin{split}
 &\Pi_{k} u(t_{n}^-)=u(t_{n}^-),\quad \textrm{for}\quad n\geq 0,\\
 &\int_{I_{n}}\big( \Pi_{k} u(t)- u(t)  \big)\chi \dd t=0,
  \quad \textrm{for}\quad \chi\in\IP_{\qq-1}, 
  \quad n\geq 1, 
\end{split}
\end{equation}
where the latter condition is not used for $\qq=0$. 
By standard arguments we then have 
\begin{equation}\label{interpolation error}
\int_{I_{n}}\| \Pi_{k} u-u\|_{j}\dd t
 \leq Ck_{n}^{\qq+1}\int_{I_{n}}\| u^{(\qq+1)}\|_{j} \dd t,  
   \quad \textrm{for} \quad j=0,1,
\end{equation}
where $ u^{(\qq)}=\frac{\partial^{\qq}u}{\partial t^{\qq}} $, 
see \cite{Quarteroni}.

First we prove a priori error estimates for the dG($\qq$) 
approximation solution at the nodal points, for which it is 
enough to use the stability estimate 
\eqref{ stability estimate}. Then, for uniform in time a 
priori error estimates, we need to use all information 
about the energy in the system, that is we need to use the 
energy identity \eqref{ stability identity}. 
We note that our analysis is limited to $\qq=0,1$  
to use the linearity property of the basis function to be able 
to prove uniform in time error estimates, since the semigroup is not analytic. 


\subsection{Estimates at the nodes}
\begin{thm}\label{Theorem2}
Let $ (U_{1},U_{2}) $ and $ (u_{1},u_{2})$ be the solutions of 
$\eqref{B(U,V)}$ and $ \eqref{B(u,v)}$, respectively.
Then with $e=(e_{1},e_{2})= (U_{1},U_{2})-(u_{1},u_{2})$ and 
for some constant $ C> 0$ (independent of $T$), we have
\begin{align}\label{Error1}
\| e_{1,N}^-\|_{1}+ \| e_{2,N}^-\|
\leq C\sum_{n=1}^{N} k_{n}^{\qq+1}\int_{I_{n}}\big\{\| u_{2}^{(\qq+1)}\|_{1}
 +\| u_{1}^{(\qq+1)}\|_{2}\big\} \dd t ,
\end{align}
\begin{align}\label{Error2}
 \hskip -1.7cm\| e_{1,N}^-\| 
  \leq C\sum_{n=1}^{N} k_{n}^{\qq+1}\int_{I_{n}}\big\{\| u_{2}^{(\qq+1)}\|
   +\| u_{1}^{(\qq+1)}\|_{1}\big\} \dd t.
\end{align}
\end{thm}
\begin{proof}
1. We split the error into two terms, recalling the interpolation 
operator $\Pi_{k}$ in \eqref{interpolation operator},  
\begin{equation*}
\begin{split}
 e
 &=(e_{1},e_{2})=(U_{1},U_{2})-(u_{1},u_{2})\\
 &=\big((U_{1},U_{2})-(\Pi_{k} u_{1},\Pi_{k} u_{2})\big)
  +\big((\Pi_{k} u_{1},\Pi_{k} u_{2})-(u_{1},u_{2})\big)\\
 &= (\theta_{1},\theta_{2})+(\eta_{1},\eta_{2})
 =\theta+\eta.
\end{split}
\end{equation*}
We can estimate the interpolation error $\eta $  
by \eqref{interpolation error}, so we need to 
find estimates for $\theta $. 
Recalling Galerkin orthogonality \eqref{GalerkinOrthogonality}, 
we have 
 \begin{align*}
  B(\theta,V)=-B(\eta,V), 
  \quad \forall V=(V_1,V_2)\in\V_{\qq}\times\V_{\qq}.
 \end{align*}
Then, using the alternative expression \eqref{B^*}, we have
\begin{equation*}
\begin{split}
B(\theta,V)
 &=-B(\eta,V)=-B^{*}(\eta,V)\\
 &=\sum_{n=1}^{N}\int_{I_{n}}\Big\lbrace(\eta_{1},\dot{V}_{1})
  +a(\eta_{2},V_{1})+( \eta_{2},\dot{V}_{2})
  -a(\eta_{1},V_{2})\Big\rbrace \dd t\\ 
 &\hskip .8cm +\sum_{n=1}^{N-1}\Big\{a(\eta^{-}_{1,n},\left[V_{1}\right]_{n} )
  +( \eta^{-}_{2,n},\left[V_{2}\right]_{n})\Big\}
 -a(\eta^{-}_{1,N},V^{-}_{1,N})
  -( \eta^{-}_{2,N},V^{-}_{2,N}).
\end{split}
\end{equation*}
Now, by the fact that $ \eta_{i} $ $ (i=1,2) $ vanishes at the 
time nodes and using the definition of $\Pi_{k} $, it follows that 
$ \dot{V}_{1}$ and $\dot{V}_{2}$ are zero or constants 
on $ I_{n} $ and hence they are orthogonal to the interpolation error.
We conclude that $ \theta=(\theta_1,\theta_2) \in \V_{\qq}\times \V_{\qq}$ satisfies the equation
\begin{align}\label{B(theta,V)}
  B(\theta,V)=\int_{0}^{t_{N}}\Big\lbrace a(\eta_{2},V_{1}) -(A\eta_{1},V_{2})\Big\rbrace \dd t .  
\end{align}
That is,  $ \theta $ satisfies  \eqref{B= L^V} with $f_{1}=\eta_{2} $ 
and $ f_{2}=-A\eta_{1}$.

2. Then applying the stability estimate
\eqref{ stability estimate} and recalling $ \theta_{i,0}=\theta_{i}(0)=0$, we have
\begin{equation} \label{theta_1+theta_2}
\begin{split}
 \| \theta_{1,N}^-\|_{l+1}+\| \theta_{2,N}^-\|_l
 &\leq C \Big ( \| \theta_{1,0}\|_{l+1}
   +\| \theta_{2,0}\|_l
   +\int_{0}^{T}\lbrace\|\eta_{2}\| _{l+1}
   +\| A\eta_{1}\|_l\rbrace \dd t \Big)\\
 &= C  \int_{0}^{T}\left\lbrace \| \eta_{2}\|_{l+1}
   +\| A\eta_{1}\|_l\right\rbrace  \dd t.
\end{split}
\end{equation}
To prove the first a priori error estimate \eqref{Error1}, we set 
$ l=0 $.
In view of $ e=\theta+\eta $ and $ \eta_{i,N}^-=0 $, we have 
\begin{align*}
  \| e_{1,N}^-\|_{1}+\| e_{2,N}^-\| \leq C  \int_0^{T}
\big\lbrace \| \eta_2\|_1+\| A\eta_{1}\|\big\rbrace \dd t.
\end{align*}
Now, using \eqref{interpolation error} and 
$ \| Au\| = \| u\|_2 $, the first a priori error estimate \eqref{Error1} is obtained.

For the second error estimate, we choose $ l=-1 $ in \eqref{theta_1+theta_2}.
In view of  $ e=\theta+\eta $ and $ \eta_{i,N}^-=0 $, we have 
\begin{align*}
  \| e_{1,N}^-\|+\| e_{2,N}^-\|_{-1} \leq C   
\int_0^T\big\{\| \eta_2\|+\| A\eta_1\|_{-1}\big\} \dd t.
\end{align*}
Now, using \eqref{interpolation error} and by the fact that 
$ \| Au\|_{-1} = \| u\|_1 $, implies the second 
a priori error estimate \eqref{Error2}.
\end{proof}

\begin{rem} \label{ExtraLoad}
We note that \eqref{B(theta,V)}, means that  
$f_{1}=\eta_{2} $ and $  f_{2}=-A\eta_{1}$ in 
\eqref{B= L^V}, which is the reason for considering an 
extra load term in the first equation of \eqref{weak form}. 
This way, we can balance between the right operators and 
suitable norms to get optimal order of convergence with 
minimal regularity requirement on the solution. 
Indeed, in \cite{Rauch}, it has been proved that the minimal regularity 
that is required for optimal order convergence for finite element 
discretization of the wave equation is one extra derivative 
compare to the optimal order of convergence, and it cannot be relaxed. 
This means that the regularity requirement on the solution in 
our error estimates are minimal. 
This is in agreement with the error estimates for 
continuous Galerkin finite element approximation 
of second order hyperbolic problems, see, e.g., 
\cite{KovacsLarssonSaedpanah2010, LarssonSaedpanah2010, 
Saedpanah2015}. 
\end{rem}

\subsection{Interior estimates}
Now, we prove uniform in time a priori error estimates for 
dG($\qq$), $\qq=0,1$, based on the linearity of the basis functions. 
\begin{thm}\label{Theorem3}
Let 
$(U_{1},U_{2}) $ and $  (u_{1},u_{2})$
 be the solutions of $  \eqref{B(U,V)}$ and $ \eqref{B(u,v)} $,  respectively.
Then with $  e=(e_{1},e_{2})= (U_{1},U_{2})-(u_{1},u_{2})$ and 
for some constant $ C> 0$ (independent of $T$), we have
\begin{equation}\label{1-E_IN}
\begin{split}
\| e_1\|_{1,J_N}+ \| e_2\|_{J_N}
  &\leq C \Big(k^{\qq+1}\| u_1^{(\qq+1)}\|_{1,J_N} 
   + k^{\qq+1}\| u_2^{(\qq+1)}\|_{J_N}\\
  &\qquad+ \sum_{n=1}^N k_n^{\qq+2}\| u_2^{(\qq+1)}\|_{1,I_n}
   +\sum_{n=1}^N k_n^{\qq+2}\| u_1^{(\qq+1)}\|_{2,I_n}\Big),
\end{split}
\end{equation}
\begin{equation}\label{2-E_IN}
\begin{split}
 \| e_1\|_{J_N}
  \leq C \Big(k^{\qq+1}\| u_1^{(\qq+1)}\|_{J_N}
   &+\sum_{n=1}^N k_n^{\qq+2}\| u_2^{(\qq+1)}\|_{I_n} \\
   &+\sum_{n=1}^N k_n^{\qq+2}\| u_1^{(\qq+1)}\|_{1,I_n}\Big).
\end{split}
\end{equation}
\end{thm}  
\begin{proof}
1. We split the error into two terms, recalling the interpolation 
operator $\Pi_{k}$ in \eqref{interpolation operator},  
\begin{equation*}
\begin{split}
 e
 &=(e_{1},e_{2})=(U_{1},U_{2})-(u_{1},u_{2})\\
 &=\big((U_{1},U_{2})-(\Pi_{k} u_{1},\Pi_{k} u_{2})\big)
  +\big((\Pi_{k} u_{1},\Pi_{k} u_{2})-(u_{1},u_{2})\big)\\
 &= (\theta_{1},\theta_{2})+(\eta_{1},\eta_{2})
 =\theta+\eta.
\end{split}
\end{equation*}
We can estimate $\eta $  by \eqref{interpolation error}, so we need 
to find estimates for $\theta $. 
Then, similar to the first part of the proof of 
Theorem \ref{Theorem2}, we obtain the equation \eqref{B(theta,V)}. 
That is,  $ \theta $ satisfies  \eqref{B= L^V} with $f_{1}=\eta_{2} $ 
and $ f_{2}=-A\eta_{1}$.

2. Then, using the energy identity \eqref{ stability identity} 
and recalling $ \theta_{i,0}=\theta_{i}(0)=0$, we can write, 
for $1\leq M \leq N$,
\begin{equation*} 
\begin{split}
  \| \theta_{1,M}^{-}\|_{l+1}^2
  &+\| \theta_{1,0}^{+} \|_{l+1}^2
   +\|  \theta_{2,M}^{-} \|_{l}^2
   +\| \theta_{2,0}^{+} \|_{l}^2 
 +\sum_{n=1}^{M-1}
   \Big\{ \|[\theta_1]_n \|_{l+1}^2
   +\|[\theta_2]_n \|_{l}^2\Big\}\\
 &= 2 \int_0^{t_{M}} 
   \big\{a(\eta_2,A^l\theta_1)
         -(A\eta_1,A^l\theta_2) 
   \big\} \\
 &\leq C\Big\{\int_0^{t_M}\|\eta_2\|_{l+1}   
   \| \theta_1\|_{l+1} \dd t
  +\int_0^{t_M}\| A\eta_1\|_l 
   \|\theta_2\|_l \dd t \Big\}\\
 &\leq C\Big\{\int_0^{t_M}\| \eta_2\|_{l+1} \dd t  
   \| \theta_1\|_{l+1,J_M}
  +\int_{0}^{t_M}\|A\eta_1\|_l \dd t 
   \|\theta_2\|_{l,J_M}\Big\}, 
\end{split}
\end{equation*}
where, Cauchy-Schwarz inequality was used. 
This implies
\begin{equation}\label{(V_i=theta_i)}
\begin{split}
  \| \theta_{1,M}^{-}\|_{l+1}^2
  &+\| \theta_{1,0}^{+} \|_{l+1}^2
   +\|  \theta_{2,M}^{-} \|_{l}^2
   +\| \theta_{2,0}^{+} \|_{l}^2
 +\sum_{n=1}^{M-1}
   \Big\{ \|[\theta_1]_n \|_{l+1}^2
   +\|[\theta_2]_n \|_{l}^2\Big\}\\
 &\leq C\Big\{\int_0^{t_N}\| \eta_2\|_{l+1} \dd t  
   \| \theta_1\|_{l+1,J_N}
  +\int_0^{t_N}\|A\eta_{1}\|_{l} \dd t 
   \|\theta_2\|_{l,J_N}\Big\}.
\end{split}
\end{equation}

Since $ \qq= 0,1 $, we have 
\begin{equation*}
\begin{split}
 \| \theta_1\|_{l+1,J_N}
 &\leq \max_{1\leq n\leq N} \Big(\| \theta_{1,n}^{-}\|_{l+1}
   +\| \theta_{1,n-1}^{+}\|_{l+1} \Big)\\
 &\leq \max_{1\leq n\leq N} \| \theta_{1,n}^{-}\|_{l+1}
   +\max_{1\leq n\leq N} \| \theta_{1,n-1}^{+}\|_{l+1}\\
 &\leq\max _{1\leq n\leq N} \| \theta_{1,n}^{-}\|_{l+1}
   +\max_{1\leq n\leq N} \Big( \| [ \theta_{1}]_{n-1} \|_{l+1}
   +\| \theta_{1,n-1}^{-}\|_{l+1}\Big)\\
 &\leq \max_{1\leq n\leq N}\| \theta_{1,n}^{-}\|_{l+1}
   +\max_{1\leq n\leq N-1} \Big( \| [ \theta_1]_n \|_{l+1}
   + \| \theta_{1,n}^{-}\|_{l+1}\Big)+\| \theta_{1,0}^{+}\|_{l+1}\\ 
 & \leq 2 \max_{1\leq n\leq N} \| \theta_{1,n}^{-}\|_{l+1}
   +\max_{1\leq n\leq N-1} \| [ \theta_1]_n \|_{l+1} 
   +\| \theta_{1,0}^{+}\|_{l+1}.
\end{split}
\end{equation*}
Note that $ \| \theta_{1,0}^{-}\|_{l+1}=\|  U_{1,0}^{-}-\Pi_{k} u_{1,0}\|_{l+1} =0 $ and hence 
\begin{align}\label{theta_1-J_N}
  \| \theta_1 \|_{l+1,J_N}^2
   \leq C \max _{1\leq n\leq N}\Big(\| \theta_{1,n}^{-}\|_{l+1}^2
    +\sum_{n=1}^{N-1} \| [\theta_1]_n\|_{l+1}^2
    +\| \theta_{1,0}^{+}\|_{l+1}^2\Big),
\end{align}
and in a similar way for $ \|\theta_2\|_{l,J_N} $, we have
\begin{align} \label{theta_2-J_N}
  \| \theta_2 \|_{l,J_N}^2
  \leq C \max _{1\leq n\leq N}\Big(\| \theta_{2,n}^{-}\|_l^2
    +\sum_{n=1}^{N-1} \| [\theta_{2}]_n\|_l^2
    +\| \theta_{2,0}^{+}\|_l^2\Big). 
\end{align}
Now, using \eqref{theta_1-J_N} and \eqref{theta_2-J_N} in 
\eqref{(V_i=theta_i)} and the fact that 
$ab\leq \frac{1}{4\epsilon}a^2+\epsilon b^2 $ for some $\epsilon>0$, 
we have 
\begin{equation*}
\begin{split}
  \| \theta_1 \|_{l+1,J_N}^2+ \| \theta_2 \|_{l,J_N}^2
  &\leq C\Big\{\int_0^{t_N} \| \eta_2\|_{l+1} \dd t 
    \|\theta_1 \|_{l+1,J_N}
    +\int_0^{t_N} \| A\eta_1\|_{l} \dd t \|\theta_2 \|_{l,J_N}\Big\}\\
  &\leq C\bigg\{ 
   \frac{1}{4\epsilon}\Big(\int_0^{t_N}\|\eta_2\|_{l+1}\dd t \Big)^2
    +\epsilon\| \theta_1 \|_{l+1,J_N}^2\\
    &\qquad+\frac{1}{4\epsilon}\Big(\int_0^{t_N}\| A\eta_1\|_{l} \dd t \Big)^2
    +\epsilon\| \theta_2 \|_{l,J_N}^2 \bigg\},
\end{split}
\end{equation*}
and as a result, we obtain
\begin{equation*}
  \| \theta_1 \|_{l+1,J_N}^2+ \| \theta_2 \|_{l,J_N}^2
   \leq C \Big\{\int_0^{t_N}\| \eta_2\|_{l+1} \dd t   
    +\int_0^{t_N}\| A\eta_1\|_{l} \dd t \Big\}^2,
\end{equation*}
that implies
\begin{equation}\label{theta_I_{N}}
\begin{split}
  \| \theta_1 \|_{l+1,J_N}+ \| \theta_2 \|_{l,J_N}
   \leq  C\Big\{\int_{0}^{t_N}\| \eta_2\|_{l+1} \dd t 
    +\int_0^{t_N}\| A\eta_1\|_l \dd t  \Big\}.   
\end{split}
\end{equation}

To prove the first a priori error estimate \eqref{1-E_IN}, 
we set $l=0$. In view of $ e=\theta+\eta $, we have  
\begin{equation*}
  \| e_1\|_{1,J_N} + \| e_2\|_{J_N} 
    \leq 
   \| \eta_1\|_{1,J_N} + \| \eta_2\|_{J_N}
    + C\Big\{\int_0^{t_N}\| \eta_2\|_1 \dd t 
    + \int_0^{t_N}\| A\eta_1\| \dd t  \Big\}.
\end{equation*}
Now, using \eqref{interpolation error}, we have
\begin{equation*}
\begin{split}
 &\int_0^{t_N}\| \eta_2\|_1 \dd t 
  =\sum_{n=1}^N\int_{I_n}\| \eta_2\|_1 \dd t  
  \leq \sum_{n=1}^N k_n^{\qq+2}\| u_2^{(\qq+1)}\|_{1,I_n}, \\
 &\int_0^{t_N}\| A\eta_1\|\dd t
  =\sum_{n=1}^N\int_{I_n}\| A\eta_1\| \dd t   
  \leq \sum_{n=1}^N k_n^{\qq+2}\| Au_1^{(\qq+1)}\|_{I_n},
\end{split}
\end{equation*}
that, having $ \| Au\| = \| u\|_{2} $, the first a priori error estimate \eqref{1-E_IN} is obtained.

For the second error estimate, we choose $ l=-1 $ in \eqref{theta_I_{N}}.
In view of $ e=\theta+\eta$, we have 
\begin{equation*}
  \| e_{1}\|_{J_N}
   \leq \| \eta_1\|_{J_N} + C\Big\{\int_0^{t_N}\| \eta_2\| \dd t 
    +\int_0^{t_N}\| A\eta_1\|_{-1} \dd t \Big\}.
\end{equation*}
Now, using \eqref{interpolation error} and by the facts that 
$ \| Au\|_{-1} = \| u\|_{1}$, implies the second 
a priori error estimate \eqref{2-E_IN}.
\end{proof}

\begin{rem} \label{Remark_4}
We note that in the second step of the proof of Theorem 
\ref{Theorem2} it was enough to use the stability estimate 
\eqref{ stability estimate}. But for uniform in time a 
priori error estimates \eqref{1-E_IN}-\eqref{2-E_IN} we 
need to use all information about the jump terms, 
and therefore we used the energy identity 
\eqref{ stability identity} in the second step of 
Theorem \ref{Theorem3}. 
\end{rem}

\begin{rem}   \label{A tilde: section 4}
Theorem \ref{Theorem2} and Theorem \ref{Theorem3}, 
recalling Remark \ref{A tilde: section 3}, hold true for 
the dG($\qq$) approximation of \eqref{Wave-eq-tildeA}, 
using the corresponding norms  
$\vertiii{\cdot}_s$ and $\vertiii{\cdot}_{s,J_N}$, 
instead of $\|\cdot\|_s$ and $\|\cdot\|_{s,J_N}$, respectively. 
\end{rem}

\section{Full discretization} \label{FullDiscrete}
In this section we study  
full discretization of \eqref{Wave-eq}
by combining dG($\qq$), $\qq=0,1$ in time and 
continuous Galerkin method of order $ \rr\geq 1 $, cG($\rr$) 
in space, to be called dG($\qq$)-cG($\rr$). 
Then, we prove a stability identity and a stability estimate 
of the full discrete method. 
We use a combination of the idea in section 
\ref{A priori} with a priori error analysis for 
continuous Galerkin finite element 
approximation in \cite{KovacsLarssonSaedpanah2010}. 
This idea was used in the context of continuous Galerkin 
approximation (only cG(1)-cG(1) in time and space) of some second 
order hyperbolic integro-differential equations,  
with applications in linear/fractional order viscoelasticity, 
see \cite{LarssonSaedpanah2010, Saedpanah2015}.

\subsection{dG($ \qq $)-cG($ \rr $) formulation}
Let $ S_{h}\subset \V =\dot{H}^{1}(\Omega) $ be a family of 
finite element spaces of continuous piecewise 
polynomials of degree at most $ \rr \geq 1$, 
with $ h $ denoting the maximum diameter of the elements. 

To apply dG($\qq$) method to formulate the full discrete 
dG($\qq$)-cG($\rr$), recalling the notation in section \ref{dG}, 
we let 
$\IP_{\qq}=\IP_{\qq}(S_{h})
 =\{  v: v(t)=\sum_{j=0}^{\qq}v_{j}t^{j}, v_{j}\in S_{h}\} $. 
For each time subinterval $ I_{n} $ we denote $ S^{n}_{h}$, 
and define the finite element spaces  
$  \V_{h,\qq}=\V_{\qq}(S_{h}) =\{ v:
 v|_{S_n}\in \IP_{\qq}( S^{n}_{h} ), \ 
 n=1, \dots, N\} $. 
We note that 
$\V_{h,\qq} \subset \V_{\qq} \subset \W$
and therefore we use the framework in section \ref{dG}. 
We denote the full discrete approximate solution by 
$U=(U_1,U_2)$, too. 

Then $ U=(U_{1},U_{2})\in \V_{h,\qq}\times\V_{h,\qq}$, 
the solution of dG($\qq$)-cG($\rr$), satisfies 
\begin{equation}\label{B(U,V)-Full}
\begin{split}
 &B(U,V)=L(V), \qquad\forall V=(V_{1},V_{2}) \in \V_{h,\qq}\times\V_{h,\qq},\\
 &U^{-}_{0}= U_{h,0},
\end{split}
\end{equation}
where $U_{h,0}=(U^{-}_{1,0},U^{-}_{2,0})=(u_{h,0},v_{h,0})$, and $u_{h,0}$ and $v_{h,0}$ are suitable 
approximations (to be chosen) of the initial data 
$u_{0}$ and $v_{0}$ in $S_h$, respectively. 
Here, the linear form $L$ 
is defined on $\W\times\W$ by 
\begin{equation}\label{B(u,v)=L(v)-Full}
 L\big((v_{1},v_{2})\big)
   =\sum_{n=1}^{N}\int_{I_{n}}( f,v_{2}) \dd t 
   +a(u_{h,0},v^{+}_{1,0})+( v_{h,0},v_{2,0}^{+}).
\end{equation}
This and  \eqref{B(u,v)} imply,  
for the error $ e=(e_{1},e_{2}) =(U_1,U_2)-(u_1,u_2)$, 
\begin{equation*} 
 B(e,V) = 
 a\big((u_{h,0}-u_0),v^{+}_{1,0})\big) 
  + \big( (v_{h,0}-v_0),v_{2,0}^{+}\big), 
  \qquad  \forall  V=(V_{1},V_{2}) \in \V_{h,\qq}\times\V_{h,\qq}.
\end{equation*}
Therefore, using the natural choice
\begin{equation} \label{IC choice}
 U^{-}_{1,0}=u_{h}^{0}=\R_h u_{0}, 
 \quad U^{-}_{2,0}=v^{0}_{h}=\PP_h v_{0},
\end{equation}
we have the Galerkin orthogonality
\begin{equation} \label{GalerkinOrthogonality-Full}
 B(e,V) = 0, 
  \qquad  \forall V=(V_{1},V_{2}) \in \V_{h,\qq}\times\V_{h,\qq}.
\end{equation}
Here, the orthogonal projections 
$\R_{h,n}: \V \rightarrow S^{n}_{h}$ and
$ \PP_{h,n}: H \rightarrow S^{n}_{h}$ are defined, respectively, by 
\begin{equation}\label{omega,theta}
\begin{split}
 a&(\R_{h,n}v-v,\chi)=0, 
  \qquad\forall v\in \V, \ \chi\in S^{n}_{h},\\
 &(\PP_{h,n}v-v,\chi)=0, 
  \qquad\forall v\in H, \ \chi\in S^{n}_{h}. 
\end{split}
\end{equation}
We define $ \R_h v $ and $ \PP_h v $, 
such that $ (\R_h v)(t)=\mathcal{R}_{h,n}v(t)$ and 
$ (\PP_h v)(t)=\PP_{h,n}v(t) $, for 
$ t\in I_{n} \ (n=1,\cdots, N)$.
We have the following error estimates:
\begin{equation}\label{ErrorR_h}
\|(\R_h-I)v\|+h\|(\R_h-I)v\|_{1}\leq Ch^{s}\|v\|_{s}, 
\quad \text{for}\quad v\in H^{s}\cap \V,\quad 0\leq s \leq r,
\end{equation}
\begin{equation}\label{ErrorP_h}
 h^{-1}\|(\PP_h-I)v\|_{-1}+\|(\PP_h-I)v\|
\leq Ch^{s}\|v\|_{s}, 
\quad \text{for}\quad v\in H^{s}\cap \V,\quad 0\leq s \leq r.
\end{equation}

We define the discrete linear operator 
$ A_{n,m}: S^{m}_{h}\rightarrow S^{n}_{h} $ by 
\begin{equation*}
a(v_{m},w_{n})=(A_{n,m}v_{m},w_{m}) 
\qquad \forall v_{m} \in  S^{m}_{h}, 
\  w_{n}\in  S^{n}_{h},
\end{equation*}
and $ A_{n}=A_{n,n} $, with discrete norms 
\begin{equation*}
\| v_{n}\|_{h,l}= \|A^{l/2}_{n} v_{n}\|
=\sqrt{(v_{n},A^{l}_{n}v_{n})}, 
\qquad v_{n}\in S^{n}_{h}, \ l\in \IR.
\end{equation*}
We introduce $ A_{h} $ such that $ A_{h}v=A_{n}v $ for 
$  v\in S^{n}_{h}$. 
We note that $\PP_{h}A= A_{h}\R_h$.
\subsection{Stability}
In this section we present a stability (energy) identity 
and stability estimate, that are used in a priori error 
analysis. 
Therefore, similar to $\S \ref{dG}$, we need a stability 
identity for a slightly more general problem, that is 
$  U=(U_1,U_2)\in \V_{h,\qq}\times \V_{h,\qq}$ such that 
\begin{equation}\label{B_{h,n}= L^V}
 B(U,V)=\hat{L}(V),   \qquad  \forall  V=(V_1,V_2)\in \V_{h,\qq}\times \V_{h,\qq},
\end{equation}
where the linear form $\hat L$ 
is defined on $\W\times\W$ by
\begin{align*}
 \hat{L}((v_{1},v_{2}))
 =\sum _{n=1}^{N} \int_{I_{n}}\Big\{ a(f_{1},v_{1})
  +( f_{2},v_{2}) \Big\} \dd t
  +a( u_{h,0},v^{+}_{1,0}) + (v_{h,0},v_{2,0}^{+}).
\end{align*}
\begin{thm} \label{A tilde: stability full discrete}
Let $ U=(U_{1},U_{2}) $ be a solution of \eqref{B_{h,n}= L^V}. 
Then for any $ T>0 $ and $ l\in\IR $, we have the energy identity
\begin{equation}\label{stability identity-Full} 
\begin{split}
\| &U_{1,N}^{-}\|_{h,l+1}^{2}
 +\| U_{2,N}^{-}\|_{h,l}^{2}
 +\sum _{n=0}^{N-1}\left\lbrace \|[U_{1}]_{n}\|_{h,l+1}^{2}
  +\|[U_{2}]_{n}\|_{h,l}^{2} \right\rbrace  \\
 &= \| u_{h,0}\|_{h,l+1}^{2}+\| v_{h,0}\|_{h,l}^{2} 
   +2\int_{0}^{T}\Big\{  a(\R_h f_{1},A_{h}^{l}U_{1})
  +(\PP_h f_{2},A_{h}^{l}U_{2}) \Big\} \dd t .
\end{split}
\end{equation}
Moreover, for some constant $ C> 0$ (independent of $T$), 
we have the stability estimate 
\begin{equation}\label{stability estimate-Full}
\begin{split}
\| U_{1,N}^{-}\|_{h,l+1}
  +\| U_{2,N}^{-}\|_{h,l}
 \leq C \Big(&  \| u_{h,0}\|_{h,l+1}+\| v_{h,0}\|_{h,l} \\ 
  &+\int_{0}^{T} \left\lbrace \|\R_h f_{1}\|_{h,l+1}
 +\| \PP_h f_{2}\| _{h,l}\right\rbrace  \dd t \Big).
\end{split}
\end{equation}
\end{thm}  
\begin{proof}
We set $ V_{i}=A_{h}^{l}U_{i} $ for~ $ i=1,2 $ in \eqref{B_{h,n}= L^V} to obtain 
\begin{equation*}
\begin{split}
\frac{1}{2}\sum_{n=1}^{N}&\int_{I_{n}}
 \frac{\partial}{\partial t}\|U_{1}\|_{h,l+1}^{2}\dd t 
  +\frac{1}{2}\sum_{n=1}^{N}\int_{I_{n}}
   \frac{\partial}{\partial t} \|U_{2}\|_{h,l}^{2}\dd t \\
 &+\sum_{n=1}^{N-1}\Big\{a([U_{1}]_{n},A_{h}^{l}U_{1,n}^{+})
   +([U_{2}]_{n},A_{h}^{l}U_{2,n}^{+})\Big\}\\
 &+a(U_{1,0}^{+},A_{h}^{l}U_{1,0}^{+})
    +(U_{2,0}^{+},A_{h}^{l}U_{2,0}^{+})\\
   =\int_{0}^{T}&\Big\{ a(\R_h f_{1},A_{h}^{l}U_{1})
   +(\PP_h f_{2},A_{h}^{l}U_{2})\Big\} \dd t 
  +a(u_{h,0},A_{h}^{l}U^{+}_{1,0})+(v_{h,0},A_{h}^{l}U^{+}_{2,0}).
\end{split}
\end{equation*}
Now, similar to the proof Theorem \ref{Theorem-stability1}, 
the stability identity \eqref{stability identity-Full} and 
stability estimate \eqref{stability estimate-Full} are proved.
\end{proof}

\begin{rem}   \label{A tilde: section 5}
For the model problem \eqref{Wave-eq-tildeA}, we recall Remark 
\ref{A tilde: section 2} and we define the orthogonal projection 
$\tilde \R_{h,n}: \V \rightarrow S^{n}_{h}$  by 
\begin{equation*} 
 \tilde a(\tilde \R_{h,n}v-v,\chi)=0, 
  \qquad\forall v\in \V, \ \chi\in S^{n}_{h}. 
\end{equation*}
We define $\tilde  \R_h v $, 
such that $ (\tilde \R_h v)(t)=\tilde \R_{h,n}v(t)$, for 
$ t\in I_{n} \ (n=1,\cdots, N)$, 
and we have the following error estimates:
\begin{equation*}
\|(\tilde \R_h-I)v\|+h\vertiii{(\tilde \R_h-I)v }_{1}
\leq Ch^{s}\vertiii{v}_{s}, 
\quad \text{for}\quad v\in H^{s}\cap \V,\quad 0\leq s \leq r,
\end{equation*}

We also define the discrete linear operator 
$\tilde A_{n,m}: S^{m}_{h}\rightarrow S^{n}_{h} $ by 
\begin{equation*}
\tilde a(v_{m},w_{n})=(\tilde A_{n,m}v_{m},w_{m}) 
\qquad \forall v_{m} \in  S^{m}_{h}, 
\  w_{n}\in  S^{n}_{h},
\end{equation*}
and $ \tilde A_{n}=\tilde A_{n,n} $, with discrete norms 
\begin{equation*}
\vertiii{ v_{n}}_{h,l}= \|\tilde A^{l/2}_{n} v_{n}\|
=\sqrt{(v_{n},\tilde A^{l}_{n}v_{n})}, 
\qquad v_{n}\in S^{n}_{h}, \ l\in \IR.
\end{equation*}
We introduce $ \tilde A_{h} $ such that 
$ \tilde A_{h}v = \tilde A_{n}v $ for 
$  v\in S^{n}_{h}$. 

Now, Theorem \ref{A tilde: stability full discrete}, 
recalling Remark \ref{A tilde: section 3}, holds true for 
the dG($\qq$) approximation of \eqref{Wave-eq-tildeA}, 
with norms $\vertiii{\cdot}_{h,s}$, the energy inner product 
$\tilde a(\cdot,\cdot)$ and the operators $\tilde A_h$ and 
$\tilde\R_h$, instead of $\|\cdot\|_{h,s},\ a(\cdot,\cdot)$, $A_h$ 
and $\R_h$, respectively.

\end{rem}
\section{A priori error estimates for full dicretization} 
\label{A priori_Full}
Here we combine the idea in section \ref{A priori} 
with the approach that was used for continuous Galerkin  
finite element approximation for second order hyperbolic problems 
in \cite{KovacsLarssonSaedpanah2010, LarssonSaedpanah2010, Saedpanah2015}. This is an extension of a priori error analysis 
to dG($\qq$)-cG($\rr$) methods. 

Similar to the temporal discretization in section \ref{A priori}, 
first we prove a priori error estimates for a general 
dG($\qq$)-cG($\rr$) approximation solution at the temporal nodal 
points, for which it is enough to use the stability estimate 
\eqref{stability estimate-Full}. Then, for uniform in time a 
priori error estimates, we use the energy identity 
\eqref{stability identity-Full}. 
Our analysis is limited to $\qq=0,1$, such that we can 
use the linearity property of the basis function to prove uniform 
in time error estimates. 

\begin{rem} \label{Remark_dG_cG}
For the error analysis of continuous Galerkin time-space  
discretization of second order hyperbolic problems, 
see, e.g., \cite[Remark 3.2]{Saedpanah2015}, 
we need to assume that 
$S_h^{n-1} \subset S_h^{n}, \ n=1,\dots, N$, 
that is, we do not change the 
spatial mesh or just refine the spatial mesh from one time 
level to the next one. This limitation on the spatial mesh 
is not needed for discontinuous Galerkin approximation in time, 
i.e., dG($\qq$)-cG($\rr$). 
\end{rem}

\subsection{Estimates at the nodes}
\paragraph*{}
\begin{thm}  \label{Theorem-Full-1}
Let $ (U_1,U_2) $ and $ (u_1,u_2)$ be the solutions of 
$\eqref{B_{h,n}= L^V}$ and $ \eqref{B(u,v)}$, respectively.
Then with $  e=(e_{1},e_{2})= (U_{1},U_{2})-(u_{1},u_{2})$ and 
for some constant $ C> 0$ (independent of $T$), we have
\begin{equation}\label{Error-Full1_1(Rh,Rh)}
\begin{split}
 \| e_{1,N}^-\|_{1} + \| e_{2,N}^-\|
 &\leq C\bigg(\sum_{n=1}^N 
  k_n^{\qq+1}\int_{I_n}\big\{\| u_{2}^{(\qq+1)}\|_{1}
   +\| u_{1}^{(\qq+1)}\|_{2}\big\} \dd t \\
 &\qquad +h^{\rr}\Big\{
   \|v_0\|_{\rr}
   +\int_{0}^{T}\|\dot{u}_{2}\|_{\rr} \dd t
   +\| u_{1,N}\|_{\rr+1} + \| u_{2,N}\|_{\rr}\Big\}\bigg),
\end{split}
\end{equation}
\begin{equation}
\begin{split}\label{Error-Full1_3(Rh,Ph)}
 \| e_{1,N}^-\|
 &\leq C\Big(\sum_{n=1}^Nk_n^{\qq+1}\int_{I_n}
   \big\{\| u_{2}^{(\qq+1)}\|
     +\| u_{1}^{(\qq+1)}\|_{1}\big\} \dd t \\
 &\qquad +h^{\rr+1}\big\{\int_{0}^{T}\| u_{2}\|_{\rr+1} \dd t
    +\| u_{1,N}\|_{\rr+1}\big\}\Big).
\end{split}
\end{equation}
\end{thm}
\begin{proof}
1. We split the error as:
\begin{equation*}
\begin{split}
 e=U-u=\big(U-\Pi_{k} \Pi_{h}u\big)
       +\big(\Pi_{k} \Pi_{h}u- \Pi_{h}u\big)
       +\big( \Pi_{h}u-u\big) 
  =\theta+\eta+\omega,
\end{split}
\end{equation*}
where $ \Pi_{k} $ is the linear interpolation operator defined by 
\eqref{interpolation operator}, and $ \Pi_{h} $ (to be specified)  
is in terms of the projectors $ \R_h $ or $ \PP_h $ in 
\eqref{omega,theta}. 

2. To prove the first error estimate we choose
\begin{equation*}
 \theta_{i}=U_{i}-\Pi_{k} \R_h u_{i}, \quad \eta_{i}
 =(\Pi_{k}-I)\R_h u_{i},\quad \omega_{i}
 =(\R_h-I)u_{i},\quad  i=1,2.
\end{equation*}
Therefore, using $ \theta=e-\eta-\omega $ and the Galerkin
orthogonality \eqref{GalerkinOrthogonality-Full}, we get
 \begin{align*}
  B(\theta,V)
  =-B(\eta,V)-B(\omega,V), 
  \quad \forall  V=(V_1,V_2)\in \V_{h,\qq}\times \V_{h,\qq}.
 \end{align*}
Then, recalling the alternative expression \eqref{B^*}, 
we have
\begin{equation*}
\begin{split}
B(\theta,V)
 &=-B(\eta,V)-B(\omega,V) 
  =-B^{*}(\eta,V)-B^{*}(\omega,V) \\
 &=\sum_{n=1}^{N}\int_{I_{n}}\Big\lbrace a(\eta_{1},\dot{V}_{1})
  +a(\eta_{2},V_{1})+( \eta_{2},\dot{V}_{2})
   -a(\eta_{1},V_{2})\Big\rbrace \dd t\\ 
 &\quad +\sum_{n=1}^{N-1}\Big\{a(\eta^{-}_{1,n},\left[V_{1}\right]_{n} )
  +( \eta^{-}_{2,n},\left[V_{2}\right]_{n})\Big\}\\
 &\quad -a(\eta^{-}_{1,N},V^{-}_{1,N})
  -( \eta^{-}_{2,N},V^{-}_{2,N})\\
  &\quad +\sum_{n=1}^{N}\int_{I_{n}}\Big\lbrace a(\omega_{1},\dot{V}_{1})
  +a(\omega_{2},V_{1})+(\omega_{2},\dot{V}_{2})
   -a(\omega_{1},V_{2})\Big\rbrace \dd t\\ 
 &\quad +\sum_{n=1}^{N-1}\Big\{a(\omega^{-}_{1,n},\left[V_{1}\right]_{n} )
  +( \omega^{-}_{2,n},\left[V_{2}\right]_{n})\Big\}\\
 &\quad -a(\omega^{-}_{1,N},V^{-}_{1,N})
  -(\omega^{-}_{2,N},V^{-}_{2,N}).
\end{split}
\end{equation*}
Now, using the definition of $\Pi_{k} $, in 
\eqref{interpolation operator} and the definition of 
$ \omega $ in \eqref{omega,theta}, we conclude that 
$  \theta=(\theta_1,\theta_2)\in \V_{h,\qq}\times \V_{h,\qq}$ satisfies the equation
 \begin{equation*}
  \begin{split}
B(\theta,V)
 &=\sum_{n=1}^{N}\int_{I_{n}}\Big\lbrace a(\eta_{2},V_{1}) 
  -a(\eta_{1},V_{2})\Big\rbrace \dd t+\sum_{n=1}^{N}\int_{I_{n}} (\omega_{2},\dot{V}_{2})\dd t\\
 &\hskip .8cm+\sum_{n=1}^{N-1}( \omega^{-}_{2,n},\left[V_{2}\right]_{n})  
  -(\omega^{-}_{2,N},V^{-}_{2,N}).
 \end{split}
\end{equation*}
Consequently, we have
\begin{equation}\label{theta-1Theorem 8.1}
B(\theta,V)
 =\sum_{n=1}^{N}\int_{I_{n}}\Big\lbrace a(\eta_{2},V_{1}) 
 -a(\eta_{1},V_{2})\Big\rbrace \dd t-\sum_{n=1}^{N}\int_{I_{n}}(\dot{\omega}_{2},V_{2})\dd t 
 -(\omega^{-}_{2,0},V^{+}_{2,0}),
\end{equation}
that is,  $ \theta $ satisfies  \eqref{B_{h,n}= L^V} with 
$f_{1}=\eta_{2}$ 
and $ f_{2}=-A\eta_{1}-\dot{\omega}_{2}$.

Applying the stability estimate 
\eqref{stability estimate-Full}, and recalling \eqref{IC choice} 
such that
\begin{equation*}
\begin{split}
 &\theta_{1,0}=\theta_{1}(0) 
  = U_1(0)-\Pi_k \R_h u_1(0) = \R_h u_0-\R_h u_0 = 0, \\
 &\theta_{2,0}=\theta_{2}(0)
  = U_2(0)-\Pi_k \R_h u_2(0) = \PP_h v_0-\R_h v_0 =(\PP_h-\R_h)v_0,
\end{split}
\end{equation*} 
we have
\begin{equation*}
\begin{split}
 \| &\theta_{1,N}^-\|_{h,l+1}+\| \theta_{2,N}^-\|_{h,l} \\
 &\leq C \Big  ( \| \theta_{1,0}\|_{h,l+1}
   +\| \theta_{2,0}\|_{h,l} 
 +\int_{0}^{T}\lbrace\|\R_h \eta_{2}\| _{h,l+1}
   +\|\PP_h \dot{\omega}_{2}\| _{h,l}
   +\| \PP_h A\eta_{1}\|_{h,l}\rbrace \dd t \Big)\\
 &=C \Big(  \| (\PP_h-\R_h)v_0 \|_{h,l} 
 +\int_{0}^{T}\left\lbrace \| \R_h \eta_{2}\|_{h,l+1}
   +\|\PP_h\dot{\omega}_{2}\| _{h,l}
   +\| \PP_h A\eta_{1}\|_{h,l}\right\rbrace  \dd t \Big).
\end{split}
\end{equation*}
Now, setting $ l=0 $ and having $ \|\cdot\|_{h,0} = \|\cdot\|$ 
and $ \|\cdot\|_{h,1} = \|\cdot\|_{1}$, we obtain
\begin{equation*}
\begin{split}
 \| \theta_{1,N}^-\|_{1}+\| \theta_{2,N}^-\|
 &\leq C \Big(  \| (\PP_h-\R_h)v_0 \| 
  +\int_{0}^{T}\left\lbrace \| \R_h \eta_{2}\|_{1}
   +\|\PP_h \dot{\omega}_{2}\| 
   +\| \PP_h A\eta_{1}\|\right\rbrace  \dd t \Big).
\end{split}
\end{equation*}
Using the fact $\|\PP_h v \| \leq\|v\| $ and 
$\|\R_h v\|_{1} \leq C \|v\|_{1} $ for all $ v\in\V $,  
and $ A_h\R_h=\PP_h A $, we have
\begin{equation*}
\begin{split}
 \| (\PP_h-\R_h)v_0 \|
 &=\| (\PP_h-\PP_h\R_h)v_0 \| 
  \leq \| (\R_h-I)v_0 \|, \\
 \| \R_h\eta_{2}\|_{1}
 & \leq C\| (\Pi_{k}-I)u_{2}\|_{1},\\
 \| \PP_h A\eta_{1}\|
 &=\| A_{h}\R_h \eta_{1}\|
  =\| (\Pi_{k}-I)A_{h}\R_h u_{1}\|
  =\| (\Pi_{k}-I)\PP_h Au_{1}\| \\
 & \leq C \| (\Pi_{k}-I) u_{1}\|_{2}.
\end{split}
\end{equation*}
In view of  $ e=\theta+\eta+\omega$ and $ \eta_{i,N}^-=0 $, we get
\begin{equation*}
\begin{split}
 \| e_{1,N}^-\|_{1} + \| e_{2,N}^-\|
 &\leq C  \Big(\| (\R_h-I)v_0 \| \\
 &\quad +\int_0^{T}
  \big\lbrace \|(\Pi_{k}-I)u_{2}\|_{1} 
   +\|(\R_h-I)\dot{u}_{2}\|
   +\| (\Pi_{k}-I) u_{1}\|_{2}\big\rbrace \dd t \\
 &\quad  +\|\omega_{1,N}^-\|_{1} +\|\omega_{2,N}^-\| \Big), 
\end{split}
\end{equation*}
that, using \eqref{interpolation error} and \eqref{ErrorR_h}, 
we imply a priori error estimate 
\eqref{Error-Full1_1(Rh,Rh)}. 

3. Finally, to prove the error estimate \eqref{Error-Full1_3(Rh,Ph)} we alter the choice as  
\begin{equation*}
\begin{split}
 &\theta_{1}
 =U_{1}-\Pi_{k} \R_h u_{1}, 
   \quad \ \ \eta_{1}=(\Pi_{k}-I)\R_h u_{1},
    \qquad \omega_{1}  =(\R_h-I)u_{1},\\
 &\theta_{2}
 =U_{2}-\Pi_{k} \PP_h u_{2}, 
  \qquad \eta_{2}=(\Pi_{k}-I)\PP_h u_{2},
   \qquad \omega_{2}=(\PP_h-I)u_{2}.
\end{split}
\end{equation*}
Now, using $ \theta=e-\eta-\omega $ and the Galerkin 
orthogonality \eqref{GalerkinOrthogonality-Full}, we have
 \begin{align*}
  B(\theta,V)=-B(\eta,V)-B(\omega,V),
  \quad \forall V=(V_1,V_2)\in \V_{h,\qq}\times \V_{h,\qq}.
 \end{align*}
Then, similar to the previous case, using the alternative expression \eqref{B^*}, we have
\begin{equation*}
\begin{split}
B(\theta,V)
 &=-B(\eta,V)-B(\omega,V)=-B^{*}(\eta,V)-B^{*}(\omega,V)\\
 &=\sum_{n=1}^{N}\int_{I_{n}}\Big\lbrace a(\eta_{1},\dot{V}_{1})
  +a(\eta_{2},V_{1})+( \eta_{2},\dot{V}_{2})
   -a(\eta_{1},V_{2})\Big\rbrace \dd t\\ 
 &\hskip .8cm +\sum_{n=1}^{N-1}\Big\{a(\eta^{-}_{1,n},\left[V_{1}\right]_{n} )
  +( \eta^{-}_{2,n},\left[V_{2}\right]_{n})\Big\}\\
 &\hskip .8cm-a(\eta^{-}_{1,N},V^{-}_{1,N})
   -( \eta^{-}_{2,N},V^{-}_{2,N})\\
 &\hskip .8cm +\sum_{n=1}^{N}\int_{I_{n}}\Big\lbrace a(\omega_{1},\dot{V}_{1})
  +a(\omega_{2},V_{1})+(\omega_{2},\dot{V}_{2})
   -a(\omega_{1},V_{2})\Big\rbrace \dd t\\ 
 &\hskip .8cm +\sum_{n=1}^{N-1}\Big\{a(\omega^{-}_{1,n},\left[V_{1}\right]_{n} )
  +( \omega^{-}_{2,n},\left[V_{2}\right]_{n})\Big\}\\
 &\hskip .8cm-a(\omega^{-}_{1,N},V^{-}_{1,N})
  -(\omega^{-}_{2,N},V^{-}_{2,N}).\\
\end{split}
\end{equation*}
Now, by the definition of $\Pi_{k} $ and 
$ \omega $, we conclude that $  \theta=(\theta_1,\theta_2)\in \V_{h,\qq}\times \V_{h,\qq}$ satisfies the equation
 \begin{equation}\label{theta-2Theorem 8.1}
  \begin{split}
  B(\theta,V)
 &=\sum_{n=1}^{N}\int_{I_{n}}\Big\lbrace a(\eta_{2},V_{1})
 -a(\eta_{1},V_{2})\Big\rbrace \dd t
 +\sum_{n=1}^{N}\int_{I_{n}} a(\omega_{2},V_{1})\dd t,
  \end{split}
\end{equation}
which is of the form \eqref{B_{h,n}= L^V} with $f_{1}=\eta_{2}+\omega_{2}$ 
and $ f_{2}=-A\eta_{1}$.

Then applying the stability estimate 
\eqref{stability estimate-Full}, and recalling 
\eqref{IC choice} such that
\begin{equation*}
\begin{split}
 &\theta_{1,0}=\theta_{1}(0) 
  = U_1(0)-\Pi_k \R_h u_1(0) = \R_h u_0-\R_h u_0 = 0, \\
 &\theta_{2,0}=\theta_{2}(0)
  = U_2(0)-\Pi_k \PP_h u_2(0) = \PP_h v_0-\PP_h v_0 = 0,
\end{split}
\end{equation*} 
we have
\begin{equation} \label{theta_1+theta_2-full1}
\begin{split}
 \| \theta_{1,N}^-\|_{h,l+1}+\| \theta_{2,N}^-\|_{h,l}
 &\leq C \Big ( \| \theta_{1,0}\|_{h,l+1}
   +\| \theta_{2,0}\|_{h,l}\\
 &\quad +\int_{0}^{T}\big\lbrace\|\R_h\eta_{2}\| _{h,l+1}+\|\R_h\omega_{2}\| _{h,l+1}
   +\| \PP_h A\eta_{1}\|_{h,l}\big\rbrace \dd t \Big)\\
 &=C  \int_{0}^{T}\big\lbrace \| \R_h\eta_{2}\|_{h,l+1}
   +\|\R_h\omega_{2}\| _{h,l+1}+\| \PP_h A\eta_{1}\|_{h,l}\big\rbrace  \dd t.
\end{split}
\end{equation}
Now, we set $ l=-1 $ and we obtain
\begin{equation*}
 \| \theta_{1,N}^-\|
 \leq C  \int_{0}^{T}\big\lbrace \| \R_h\eta_{2}\|
   +\|\R_h\omega_{2}\| +\| \PP_h A\eta_{1}\|_{h,-1}\big\rbrace  \dd t.
\end{equation*}
Then, since
\begin{equation*}
\begin{split}
&\| \R_h\eta_{2}\|=\| \R_h(\Pi_{k}-I)\PP_h u_{2}\|=\|(\Pi_{k}-I)\PP_h u_{2}\|\leq\| (\Pi_{k}-I)u_{2}\|,\\
 &\|\R_h\omega_{2}\|=\|\R_h(\PP_h-I)u_{2}\|=\|\PP_h(I-\R_h)u_{2}\|\leq\| (\R_h-I)u_{2}\|,\\
   &\| \PP_h A\eta_{1}\|_{h,-1}=\| A_h \R_h(\Pi_{k}-I)u_{1}\|_{h,-1}=\| (\Pi_{k}-I)\R_h u_{1}\|_{h,1}\\
 &\hskip 2cm\leq C\| (\Pi_{k}-I)u_{1}\|_{1},
\end{split}
\end{equation*}
in view of $ e=\theta+\eta+\omega $, $ \eta_{i,N}^-=0 $, we conclude that
\begin{align*}
  \| e_{1,N}^-\| \leq C   \Big\{\int_0^{T}
\big\lbrace \| (\Pi_{k}-I)u_{2}\|+\| (\R_h-I)u_{2}\|+\| (\Pi_{k}-I)u_{1}\|_{1}\big\rbrace \dd t+\|\omega_{1,N}^-\| \Big\}.
\end{align*}
Which implies that last estimate by \eqref{interpolation error} and \eqref{ErrorR_h}. 
The proof is now complete. 
\end{proof}

\subsection{Interior estimates}
\begin{thm}   \label{Theorem-Full-3}
Let 
$(U_{1},U_{2}) $ and $(u_{1},u_{2})$ be the solutions of 
$\eqref{B_{h,n}= L^V}$ and $ \eqref{B(u,v)} $,  respectively.
Then with $  e=(e_{1},e_{2})= (U_{1},U_{2})-(u_{1},u_{2})$ and 
for some constant $ C> 0$ (independent of $T$), 
we have
\begin{equation}\label{Error-Full-IN1_1(Rh,Rh)}
\begin{split}
 \| e_1\|_{1,J_N}
 &+ \| e_2\|_{J_N} \\
 & \leq C \Big(k^{\qq+1}\big\{\| u_1^{(\qq+1)}\|_{1,J_N} 
   +\| u_2^{(\qq+1)}\|_{1,J_N}\big\}\\
 &\qquad + \sum_{n=1}^N k_n^{\qq+2}\big\{\| u_2^{(\qq+1)}\|_{1,I_n}
 +\| u_1^{(\qq+1)}\|_{2,I_n}\big\}\\
 &\qquad+h^{\rr}\big\lbrace \|v_0\|_{\rr} 
  +\int_{0}^{T}\| \dot{u}_{2}\|_{\rr} \dd t
 +\| u_{1}\|_{\rr+1,J_N}+\| u_{2}\|_{\rr,J_N}\big\rbrace\Big),
\end{split}
\end{equation}
\begin{equation}\label{Error-Full-IN1_2(Rh,Ph)}
\begin{split}
 \| e_1\|_{J_N}
 \leq C \Big(&k^{\qq+1}\| u_1^{(\qq+1)}\|_{1,J_N}
+\sum_{n=1}^N k_n^{\qq+2} \big\{\|u_2^{(\qq+1)}\|_{I_n}+\| u_1^{(\qq+1)}\|_{1,I_n}\big\}\\
 &+h^{\rr+1}\big\lbrace\int_{0}^{T}\| u_{2}\|_{\rr+1} \dd t
  + \| u_{1}\|_{\rr+1,J_N}\big\rbrace\Big).
\end{split}
\end{equation}
\end{thm}  
\begin{proof}
1. We split the error as:
\begin{equation*}
 e=U-u=\big(U-\Pi_{k} \Pi_{h}u\big)
 +\big(\Pi_{k} \Pi_{h}u- \Pi_{h}u\big)
 +\big( \Pi_{h}u-u\big) =\theta+\eta+\omega,
\end{equation*}
where $ \Pi_{k} $ is the linear interpolation operator defined by 
\eqref{interpolation operator}, and $ \Pi_{h} $ (to be specified) 
is in terms of the projectors $ \R_h $ or $ \PP_h  $ in 
\eqref{omega,theta}. 

2. To prove the first error estimate 
\eqref{Error-Full-IN1_1(Rh,Rh)}, we choose
\begin{equation*}
  \theta_{i}=U_{i}-\Pi_{k} \R_h u_{i}, \quad \eta_{i}=(\Pi_{k}-I)\R_h u_{i},\quad \omega_{i}=(\R_h-I)u_{i},\quad  i=1,2.
\end{equation*}
Similar to the second part of the proof of Theorem 
\ref{Theorem-Full-1}, we obtain equation 
\eqref{theta-1Theorem 8.1},  
that is, $ \theta $ satisfies \eqref{B_{h,n}= L^V} 
 with $f_{1}=\eta_{2} $ and $ f_{2}=-A\eta_{1}-\dot{\omega}_{2}$.

Then, using the energy identity \eqref{stability identity-Full} 
and recalling 
\begin{equation*}
 \theta_{1,0}=\theta_{1}(0)=0,\quad 
 \theta_{2,0}=\theta_{2}(0)=(\PP_h-\R_h)v_0,
\end{equation*} 
we have, for $1\leq M \leq N$,
\begin{equation*} 
\begin{split}
 \| \theta_{1,M}^{-}&\|_{h,l+1}^2
 +\| \theta_{1,0}^{+} \|_{h,l+1}^2
   +\|  \theta_{2,M}^{-} \|_{h,l}^2
   +\| \theta_{2,0}^{+} \|_{h,l}^2\\
 &\quad+\sum_{n=1}^{M-1}
   \Big\{ \|[\theta_1]_n \|_{h,l+1}^2
   +\|[\theta_2]_n \|_{h,l}^2\Big\}\\
 &= \| (\PP_h-\R_h)v_0 \|_{h,l} \\ 
 &\quad +2 \int_0^{t_{M}} 
   \big\{a(\R_h\eta_2,A_{h}^l\theta_1)
    -(\PP_h A\eta_1,A_{h}^l\theta_2) 
    -(\PP_h \dot{\omega}_{2},A_{h}^l\theta_2) 
   \big\}\dd t \\
 &\leq \| (\PP_h-\R_h)v_0 \|_{h,l} \\ 
 &\quad +C\Big\{\int_0^{t_M}\|\R_h\eta_2\|_{h,l+1}   
   \| \theta_1\|_{h,l+1} \dd t
  +\int_0^{t_M}\|\PP_h  A\eta_1\|_{h,l} 
   \|\theta_2\|_{h,l}  \dd t \\
 &\qquad+\int_0^{t_M}\|\PP_h  \dot{\omega}_{2}\|_{h,l} 
   \|\theta_2\|_{h,l} \dd t \Big\}\\
 &\leq \| (\PP_h-\R_h)v_0 \|_{h,l} \\ 
 &\quad +C\Big\{\int_0^{t_M}\| \R_h\eta_2\|_{h,l+1} \dd t  
   \| \theta_1\|_{h,l+1,J_M}
  +\int_{0}^{t_M}\|\PP_h A\eta_1\|_{h,l} \dd t 
   \|\theta_2\|_{h,l,J_M}\\
 &\qquad +\int_{0}^{t_M}\|\PP_h  \dot{\omega}_{2}\|_{h,l} \dd t 
   \|\theta_2\|_{h,l,J_M}\Big\}, 
\end{split}
\end{equation*}
where, Cauchy-Schwarz inequality was used. 
That implies
\begin{equation}\label{(V_i=theta_i_2)}
\begin{split}
  \| &\theta_{1,M}^{-}\|_{h,l+1}^2
 +\| \theta_{1,0}^{+} \|_{h,l+1}^2
   +\|  \theta_{2,M}^{-} \|_{h,l}^2
   +\| \theta_{2,0}^{+} \|_{h,l}^2 \\
 &\quad +\sum_{n=1}^{M-1}
   \Big\{ \|[\theta_1]_n \|_{h,l+1}^2
   +\|[\theta_2]_n \|_{h,l}^2\Big\}\\
 &\leq \| (\PP_h-\R_h)v_0 \|_{h,l} \\ 
 &\quad + C\Big\{\int_0^{t_N}\| \R_h\eta_2\|_{h,l+1} \dd t\| \theta_1\|_{h,l+1,J_N}
   +\int_0^{t_N}\|\PP_h A\eta_{1}\|_{h,l}    \dd t\|\theta_2\|_{h,l,J_N}  \\ 
  &\qquad +\int_0^{t_N}\|\PP_h \dot{\omega}_{2}\|_{h,l} \dd t  \|\theta_2\|_{h,l,J_N}\Big\}.
\end{split}
\end{equation}
Since $ \qq= 0,1 $, we have 
\begin{equation*}
\begin{split}
 \| \theta_1\|_{h,l+1,J_N}
 &\leq \max_{1\leq n\leq N} \Big(\| \theta_{1,n}^{-}\|_{h,l+1}
   +\| \theta_{1,n-1}^{+}\|_{h,l+1} \Big)\\
 &\leq \max_{1\leq n\leq N} \| \theta_{1,n}^{-}\|_{h,l+1}
   +\max_{1\leq n\leq N} \| \theta_{1,n-1}^{+}\|_{h,l+1}\\
 &\leq\max _{1\leq n\leq N} \| \theta_{1,n}^{-}\|_{h,l+1}
   +\max_{1\leq n\leq N} \Big( \| [ \theta_{1}]_{n-1} \|_{h,l+1}
   +\| \theta_{1,n-1}^{-}\|_{h,l+1}\Big)\\
 &\leq \max_{1\leq n\leq N}\| \theta_{1,n}^{-}\|_{h,l+1}
   +\max_{1\leq n\leq N-1} \Big( \| [ \theta_1]_n \|_{h,l+1}
   + \| \theta_{1,n}^{-}\|_{h,l+1}\Big)\\
 &\hskip .6cm+\| \theta_{1,0}^{+}\|_{h,l+1}\\ 
 & \leq 2 \max_{1\leq n\leq N} \| \theta_{1,n}^{-}\|_{h,l+1}
   +\max_{1\leq n\leq N-1} \| [ \theta_1]_n \|_{h,l+1} 
   +\| \theta_{1,0}^{+}\|_{h,l+1}.
\end{split}
\end{equation*}
Note that $ \| \theta_{1,0}^{-}\|_{h,l+1}=\|  U_{1,0}^{-}-\Pi_{k}\R_h u_{0}\|_{h,l+1} =0 $ and hence 
\begin{align}\label{theta_1-J_N_2}
  \| \theta_1 \|_{h,l+1,J_N}^2
  \leq C \max _{1\leq n\leq N}\Big(\| \theta_{1,n}^{-}\|_{h,l+1}^2
  +\sum_{n=1}^{N-1} \| [\theta_1]_n\|_{h,l+1}^2
  +\| \theta_{1,0}^{+}\|_{h,l+1}^2\Big),
\end{align}
and in a similar way for $ \|\theta_2\|_{h,l,J_N} $, we have
\begin{align} \label{theta_2-J_N_2}
  \| \theta_2 \|_{h,l,J_N}^2
  \leq C \max _{1\leq n\leq N}\Big(\| \theta_{2,n}^{-}\|_{h,l}^2
  +\sum_{n=1}^{N-1} \| [\theta_{2}]_n\|_{h,l}^2
  +\| \theta_{2,0}^{+}\|_{h,l}^2\Big). 
\end{align}
Now, using \eqref{theta_1-J_N_2} and \eqref{theta_2-J_N_2} in 
\eqref{(V_i=theta_i_2)} and the fact that 
$ab\leq \frac{1}{4\epsilon}a^2+\epsilon b^2 $ for some $\epsilon>0$, 
we have 
\begin{equation*}
\begin{split}
  \| \theta_1 \|_{h,l+1,J_N}^2+ \| \theta_2 \|_{h,l,J_N}^2
  &\leq \| (\PP_h-\R_h)v_0 \|_{h,l} \\ 
 &\quad + C\Big\{\int_0^{t_N} \| \R_h\eta_2\|_{h,l+1} \dd t 
    \|\theta_1 \|_{h,l+1,J_N}\\
  &\qquad +\int_0^{t_N} \| \PP_h A\eta_1\|_{h,l} \dd t \|\theta_2 \|_{h,l,J_N}\\
  &\qquad +\int_0^{t_N} \| \PP_h \dot{\omega}_{2}\|_{h,l} \dd t \|\theta_2 \|_{h,l,J_N}\Big\}\\
  &\leq \| (\PP_h-\R_h)v_0 \|_{h,l} \\ 
 &\quad +C\bigg\{ 
   \frac{1}{4\epsilon}\Big(\int_0^{t_N}\|\R_h\eta_2\|_{h,l+1}\dd t \Big)^2
    +\epsilon\| \theta_1 \|_{h,l+1,J_N}^2\\
  &\qquad+\frac{1}{4\epsilon}\Big(\int_0^{t_N}\| \PP_h A\eta_1\|_{h,l} \dd t \Big)^2
    +\epsilon\| \theta_2 \|_{h,l,J_N}^2\\
  &\qquad+\frac{1}{4\epsilon}\Big(\int_0^{t_N}\| \PP_h \dot{\omega}_{2}\|_{h,l} \dd t \Big)^2
    +\epsilon\| \theta_2 \|_{h,l,J_N}^2 \bigg\},
\end{split}
\end{equation*}
and as a result, we obtain
\begin{equation*}
\begin{split}
  \| \theta_1 \|_{h,l+1,J_N}^2+ \| \theta_2 \|_{h,l,J_N}^2
 &\leq \| (\PP_h-\R_h)v_0 \|_{h,l} \\ 
 &\quad + C \Big\{\int_0^{t_N}\|\R_h \eta_2\|_{h,l+1} \dd t   
    +\int_0^{t_N}\| \PP_h A\eta_1\|_{h,l} \dd t\\
 & \qquad+\int_0^{t_N}\| \PP_h \dot{\omega}_{2}\|_{h,l} \dd t \Big\}^2,
  \end{split}
\end{equation*}
that implies
\begin{equation*}
\begin{split}
  \| \theta_1 \|_{h,l+1,J_N}+ \| \theta_2 \|_{h,l,J_N}
 & \leq  \| (\PP_h-\R_h)v_0 \|_{h,l} \\ 
 &\quad +C\Big\{\int_{0}^{t_N}\| \R_h\eta_2\|_{h,l+1} \dd t 
    +\int_0^{t_N}\|\PP_h  A\eta_1\|_{h,l} \dd t \\
  & \hskip .8cm+\int_0^{t_N}\|\PP_h \dot{\omega}_{2}\|_{h,l} \dd t  \Big\}.   
\end{split}
\end{equation*}
Now, setting $l=0$ and having  
$ \|\cdot\|_{h,0} = \|\cdot\|$
 and $ \|\cdot\|_{h,1} = \|\cdot\|_{1}$, we obtain
\begin{equation*}
  \| \theta_1 \|_{1,J_N}+ \| \theta_2 \|_{J_N}
 \leq  \| (\PP_h-\R_h)v_0 \| 
  +C\Big(\int_{0}^{t_N}\big\lbrace\| \R_h\eta_2\|_{1}
   +\|\PP_h A\eta_1\| +\|\PP_h \dot{\omega}_{2}\| \big\rbrace\dd t  \Big).   
\end{equation*}
Using the fact that 
$\|\PP_h v \|_{1} \leq C \|v\|_{1} $,  
$\|\PP_h v \| \leq\|v\| $ and 
$\|\R_h v\|_{1} \leq C \|v\|_{1} $,  for all $ v\in \V $, and $ A_h\R_h=\PP_h A $, we get
\begin{equation*}
\begin{split}
 \| (\PP_h-\R_h)v_0 \|
 &=\| (\PP_h-\PP_h\R_h)v_0 \| 
  \leq \| (\R_h-I)v_0 \|, \\
 \| \R_h\eta_{2}\|_{1}
 & \leq C\| (\Pi_{k}-I)u_{2}\|_{1},\\
 \| \PP_h A\eta_{1}\|
 &=\|A_h\R_h  \eta_{1}\|
 =\| (\Pi_{k}-I)A_h\R_h u_{1}\|
 =\| (\Pi_{k}-I)\PP_h Au_{1}\|\\
 &
 \leq C \| (\Pi_{k}-I) u_{1}\|_{2}.
\end{split}
\end{equation*}
In view of $ e=\theta+\eta+\omega $, we have
\begin{equation*}
\begin{split}
  \| e_1\|_{1,J_N}
 &+\| e_2\|_{J_N} 
 \leq  \| (\R_h-I)v_0 \| \\
 &\quad+ C\Big(\int_{0}^{t_N}\big\lbrace\| (\Pi_{k}-I)u_{2}\|_{1}
    +\|(\Pi_{k}-I)u_{1}\|_{2}+\|(\R_h-I)\dot{u}_{2}\|\big\rbrace\dd t\\
    &\hskip 2.5cm+\| \eta_1\|_{1,J_N}
   +\| \eta_2\|_{J_N}+\| \omega_1\|_{1,J_N}
   +\|\omega_2\|_{J_N}\Big).
\end{split}
\end{equation*}
Now, using \eqref{interpolation error} and \eqref{ErrorR_h} 
we conclude a priori error estimate \eqref{Error-Full-IN1_1(Rh,Rh)}.

3. To prove the second error estimate \eqref{Error-Full-IN1_2(Rh,Ph)}, we choose
\begin{equation*}
\begin{split}
 &\theta_{1}=U_{1}-\Pi_{k} \R_h u_{1}, \qquad \eta_{1}=(\Pi_{k}-I)\R_h u_{1},\qquad \omega_{1}=(\R_h-I)u_{1},\\
 &\theta_{2}=U_{2}-\Pi_{k} \PP_h u_{2}, \qquad \eta_{2}=(\Pi_{k}-I)\PP_h u_{2},\qquad \omega_{2}=(\PP_h -I)u_{2}.
\end{split}
\end{equation*}
Then, similar to the third part of the proof of 
Theorem \ref{Theorem-Full-1}, we obtain the equation \eqref{theta-2Theorem 8.1},  
that is,  $ \theta $ satisfies  \eqref{B_{h,n}= L^V} with $f_{1}=\eta_{2}+\omega_{2}$ 
and $ f_{2}=-A\eta_{1}$.

Then using the energy identity \eqref{stability identity-Full} 
and recalling $ \theta_{i,0}=\theta_{i}(0)=0$, we get 
\begin{equation*}
\begin{split}
  \| \theta_1 \|_{h,l+1,J_N}+ \| \theta_2 \|_{h,l,J_N}
 & \leq  C\Big\{\int_{0}^{t_N}\| \R_h\eta_2\|_{h,l+1} \dd t 
    +\int_0^{t_N}\| \R_h \omega_{2}\|_{h,l+1} \dd t \\
 & \hskip .8cm+\int_0^{t_N}\|\PP_h A\eta_{1}\|_{h,l} \dd t  \Big\}.   
\end{split}
\end{equation*}
Now, we set $l=-1$ and we obtain
\begin{equation*}
  \| \theta_1 \|_{J_N}
  \leq  C\Big(\int_{0}^{t_N}\big\lbrace\| \R_h\eta_2\| 
  +\| \R_h \omega_{2}\|+\|\PP_h A\eta_{1}\|_{h,-1} \big\rbrace\dd t \Big).
\end{equation*}
Then since
\begin{equation*}
\begin{split}
 &\| \R_h\eta_{2}\|=\| \R_h(\Pi_{k}-I)\PP_h u_{2}\|=\|\PP_h (\Pi_{k}-I)u_{2}\|\leq\| (\Pi_{k}-I)u_{2}\|,\\
 &\|\R_h\omega_{2}\|=\|\R_h(\PP_h -I)u_{2}\|=\|\PP_h(I-\R_h)u_{2}\|\leq\| (\R_h-I)u_{2}\|,\\
 & \| \PP_h A\eta_{1}\|_{h,-1}=\| A_h \R_h(\Pi_{k}-I)u_{1}\|_{h,-1}=\| (\Pi_{k}-I)\R_h u_{1}\|_{h,1}\\
 &\hskip 2cm\leq C\| (\Pi_{k}-I)u_{1}\|_{1}.
\end{split}
\end{equation*}
In view of $ e=\theta+\eta+\omega $, we have
\begin{equation*}
\begin{split}
  \| e_1\|_{J_N}
  &\leq C \Big(\int_{0}^{t_N}\big\lbrace\|(\Pi_{k}-I)u_{2}\|
+\| (\R_h-I)u_{2}\|+\| (\Pi_{k}-I)u_{1}\|_{1}\big\rbrace \dd t\\
&\hskip 1cm +\| \eta_1\|_{J_N}+\| \omega_1\|_{J_{N}}\Big).
\end{split}
\end{equation*}
Now, using \eqref{interpolation error} and \eqref{ErrorR_h} 
a priori error estimate \eqref{Error-Full-IN1_2(Rh,Ph)} is obtained.
\end{proof}

\begin{rem}   \label{A tilde: section 6}
Theorem \ref{Theorem-Full-1} and Theorem \ref{Theorem-Full-3}, 
recalling Remark \ref{A tilde: section 5}, hold true for 
the dG($\qq$)-cG($\rr$) approximation of \eqref{Wave-eq-tildeA}, 
using the corresponding norms  
$\vertiii{\cdot}_s$ and $\vertiii{\cdot}_{s,J_N}$, 
instead of $\|\cdot\|_s$ and $\|\cdot\|_{s,J_N}$, respectively. 
\end{rem}


\section{Numerical example} \label{Example}
In this section, we illustrate the temporal rate of convergence for 
dG(0)-cG(1) and dG(1)-cG(1), based on the uniform in time error 
estimates, by a simple example. 
We also present the pointwise (in time) error estimates and the discrete energy. 

\subsection{System of linear equations for dG(0) and dG(1) time-stepping}
For the piecewise constant case, dG(0),  
we have the system of linear equations, for $n=1,\dots,N$,  
\begin{align*}
\begin{bmatrix}
 A &-k_n A \\
 k_n A &M \\
\end{bmatrix}
\begin{bmatrix}
 U _{1,n}\\
 U_{2,n} \\
\end{bmatrix}
=
\begin{bmatrix}
 A & 0 \\
 0 & M \\
\end{bmatrix}
\begin{bmatrix}
 U _{1,n-1}\\
 U_{2,n-1} \\
\end{bmatrix}
+k_n\begin{bmatrix}
 0 \\ 
  F_n\\
\end{bmatrix},
\end{align*}
where $A$ and $M$ are the stiffness and mass matrices, respectively, 
and $F_n$ is the load vector. 

For the piecewise linear case, dG(1), we define 
$\Psi_{n} ^{1}(t)=\frac{t_{n}-t}{k_{n}}$, 
$\Psi_{n}^{2}(t)=\frac{t-t_{n-1}}{k_{n}}$ 
and use the representation, for $i=1,2$,  
\begin{equation*}
 U_{i}(x,t)
 =U_{i,n-1}^{+}(x)\Psi^{1}_{n}(t)+U_{i,n}^{-}(x)\Psi^{2}_{n}(t), 
 \quad x\in\Omega,\ t\in I_{n}.
\end{equation*} 
Then, the system of linear equations, for $n=1,\dots,N$, is
{\begin{align*}
\begin{bmatrix}
  \frac{1}{2}A & \frac{1}{2}A & -\omega_{n}^{12}A&-\omega_{n}^{11}A\\
   \frac{1}{2}A &-\frac{1}{2}A & -\omega_{n}^{22}A&-\omega_{n}^{21}A\\
   \omega_{n}^{12}A&\omega_{n}^{11}A &\frac{1}{2} M &\frac{1}{2} M \\
   \omega_{n}^{22}A&\omega_{n}^{21}A &\frac{1}{2} M&-\frac{1}{2} M
\end{bmatrix}
&\begin{bmatrix}
  U _{1,n}^{-}\\
  U_{1,n-1}^{+} \\
  U _{2,n}^{-}\\
  U_{2,n-1}^{+}
\end{bmatrix}\\
& =
\begin{bmatrix}
  A & 0 & 0 & 0 \\
  0 & 0 & 0 & 0\\
  0 & 0 & M & 0 \\
  0 & 0 & 0 & 0
\end{bmatrix}
\begin{bmatrix}
  U _{1,n-1}^{-}\\
  U_{1,n-2}^{+} \\
  U _{2,n-1}^{-}\\
  U_{2,n-2}^{+}
\end{bmatrix}
+\begin{bmatrix}
  0\\
  0\\
  F_{n1}\\
  F_{n2}
\end{bmatrix},
\end{align*}} 
where 
$\omega _{n}^{pr}
=\int_{I_{n}}\Psi_{n}^{r}(t)\Psi_{n}^{p}(t) \dd t$, 
and $F_{np},\ p=1,2$ are the load vectors with components 
$ f_{np}=\int_{I_{n}}( f(t),\Psi_{n}^{p}(t))\dd t $. 

\subsection{Example} 
We consider  \eqref{Wave-eq}  in one dimension  with 
homogeneous Dirichlet boundary condition, the 
source term $ f=0$, and the initial conditions 
$u(x,0) =\sin x$ and $ \dot{u}(x,0) =0$, 
for which the exact solution is 
$u(x,t)=\sin x \cos t$, $x \in [0,\pi],\ t\in[0,T]$. 

Figure 1 shows the optimal rate of convergence for dG(0)-cG(1) 
and dG(1)-cG(1) with  uniform in time $L_2$-norm for the displacement and the velocity, 
that is in agreement with \eqref{Error-Full-IN1_1(Rh,Rh)} 
and \eqref{Error-Full-IN1_2(Rh,Ph)}. 
We have used short time $T=1$ and long time $T=10$. 
The figures for the error estimates \eqref{Error-Full1_1(Rh,Rh)} 
and \eqref{Error-Full1_3(Rh,Ph)} are very similar, as expected, 
and therefore they are not presented here. 

The behaviour of the errors $\|e_1(t)\|$ and $\|e_2(t)\|$ in time 
are shown in Figure 2 for both dG(0) and dG(1), with $T=10$. 
We note that the slope of error accumulation is very small, 
in particular, for dG(1) in compare with the finest time mesh 
$10*2^{-9}$. This shows that the error accumulation does not 
depend on time $T$ in long-time integration, 
since the the stability constants are independent of $T$. 

The total energy of the system is  
$  \frac12\|\nabla u\|^2 + \frac12\|\dot u \|^2 = \frac{\pi}{4}$. 
The discrete energy (for three different time steps) has been 
compared with the theoretical energy in Figure 3, for both methods 
dG(0) and dG(1).  
In Figure 4, we show the discrete energy for dG(0) with 
the time step $k=10\cdot 2^{-9}$, and for dG(1) with even a  
bigger time step $k=10\cdot 2^{-7}$. 
In all experiments dG(1) outperforms dG(0). 
It is shown that dG(1) is much more accurate even with a 
considerably larger time step. 

\begin{figure}[htp]  \label{Fig: Order}
 \centering
 \includegraphics[width=6.2cm,height=4.1cm]
   {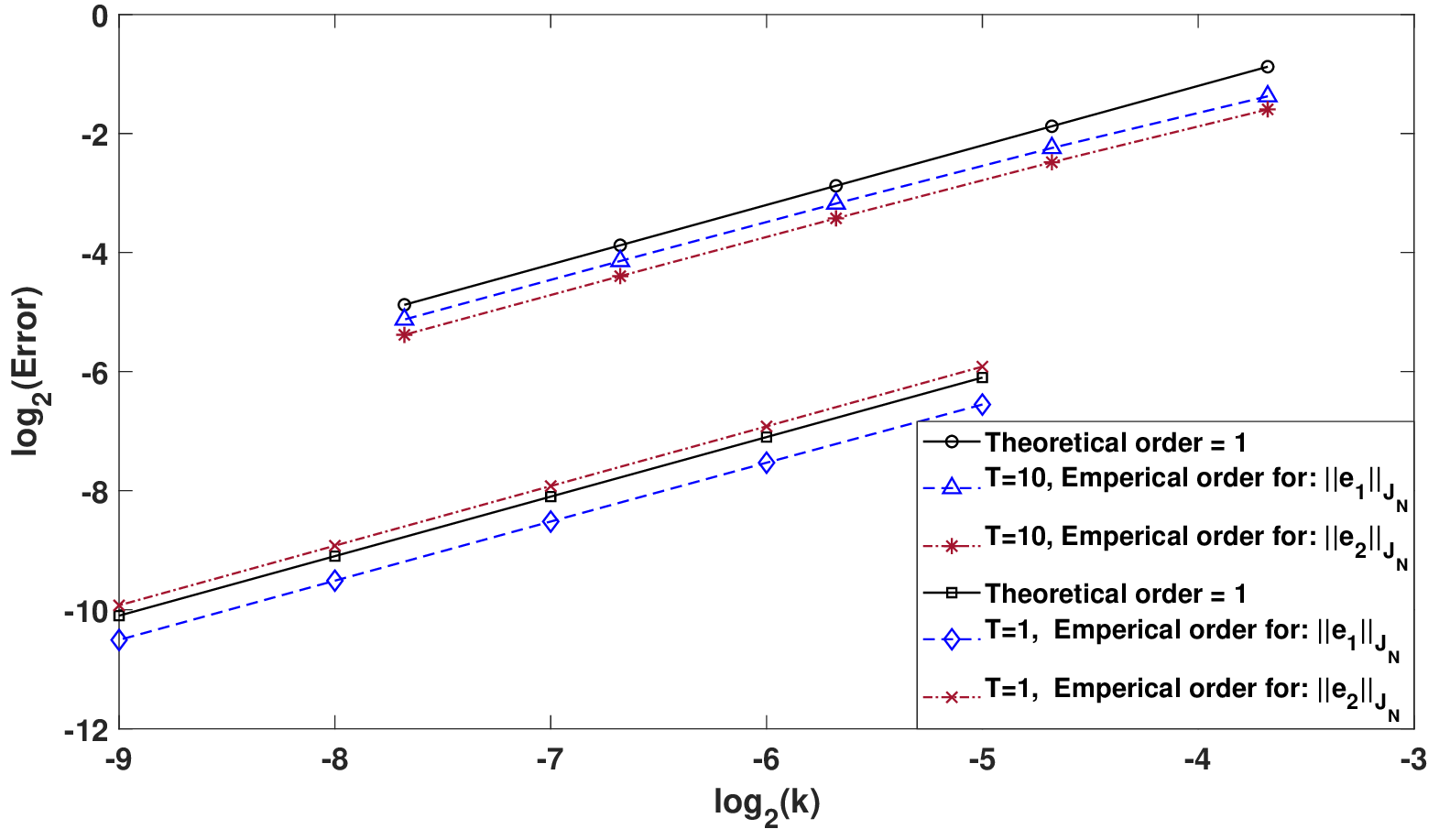}
 \includegraphics[width=6.2cm,height=4.1cm]
   {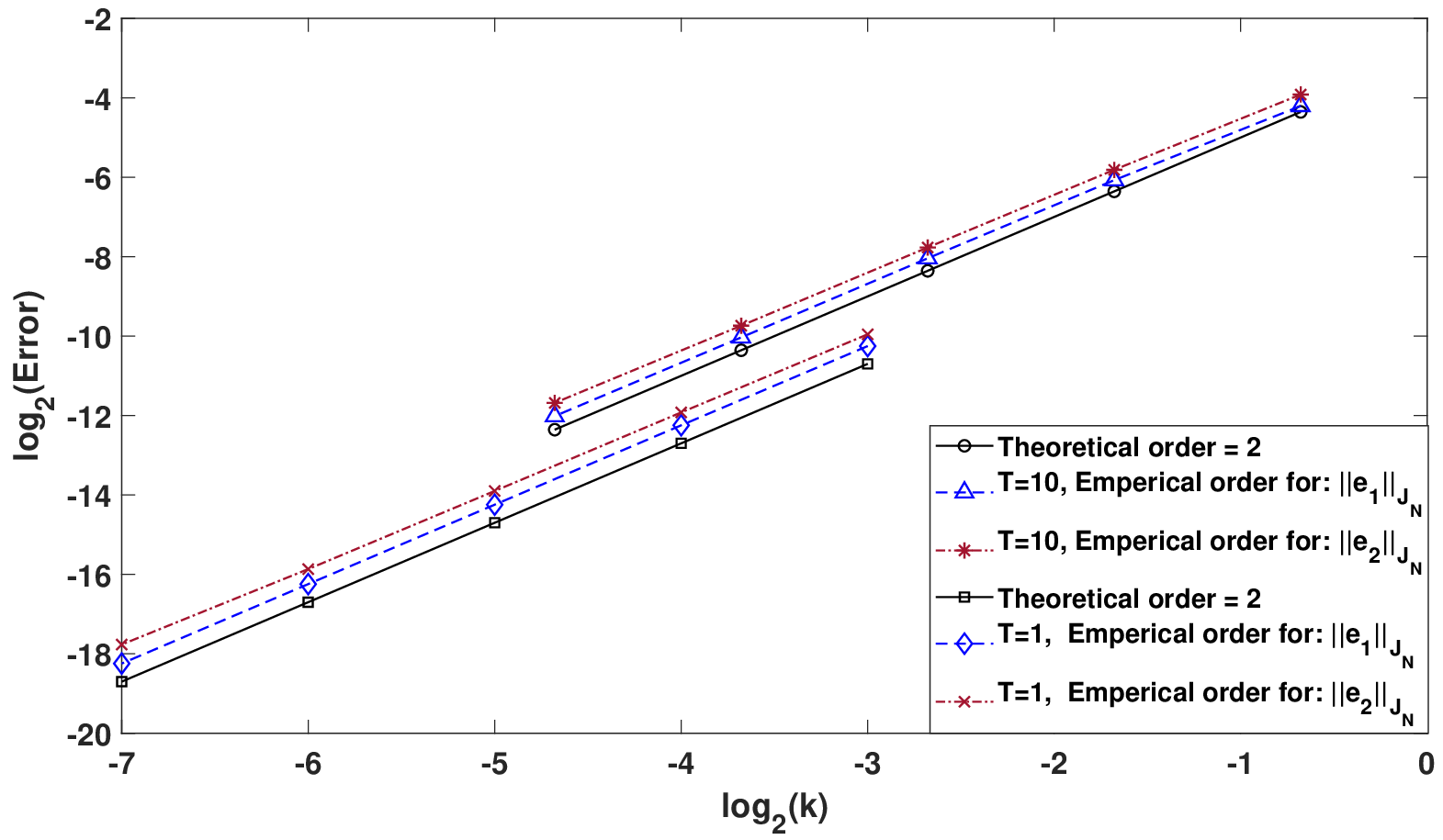}
 \caption{Temporal rate of convergence with uniform in time 
 $L^{2}$-norm for the displacement and the velocity, with $T=1, 10$: 
 (left) dG(0) (right) dG(1).}
\end{figure}
\begin{figure}[htp]  \label{Fig: error behaviour}
 \centering
 \includegraphics[width=12.cm,height=3.1cm]
   {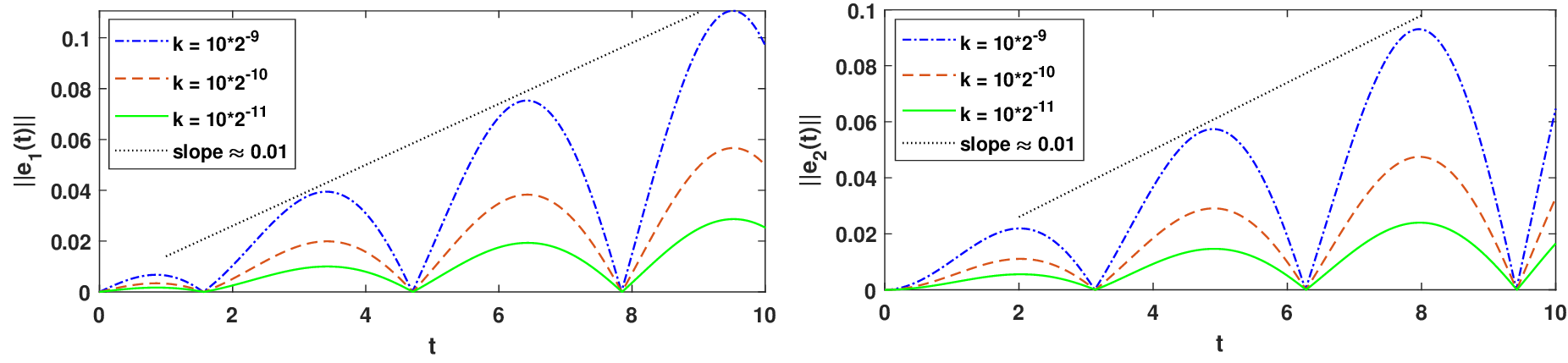}
 \includegraphics[width=12.cm,height=3.1cm]
   {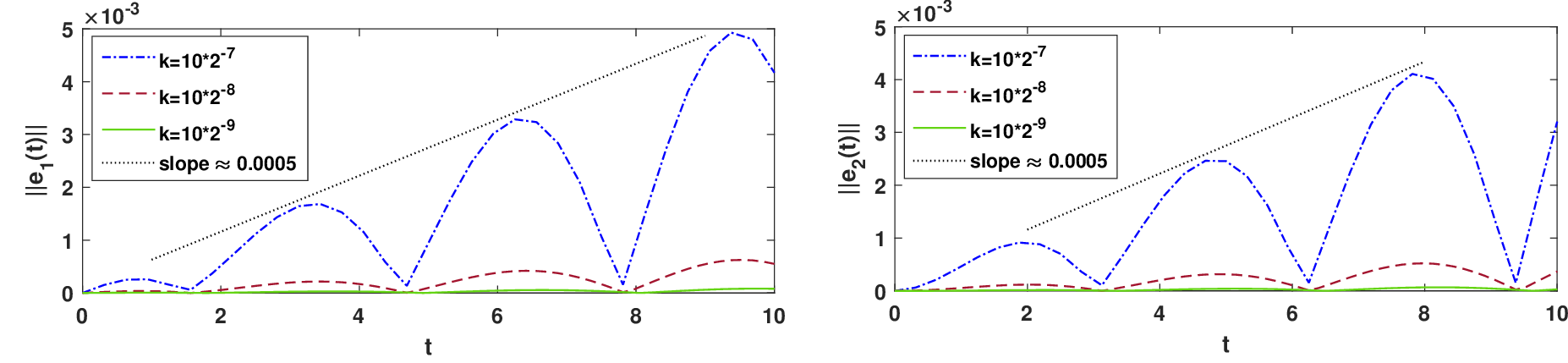}
 \caption{Behaviour of the errors $\|e_1(t)\|$ and $\|e_2(t)\|$ 
 in time: (up) dG(0) (down) dG(1).}
\end{figure}
\begin{figure}[htp]  \label{Fig: Discrete energy}
 \centering
 \includegraphics[width=6.1cm,height=3.2cm]
   {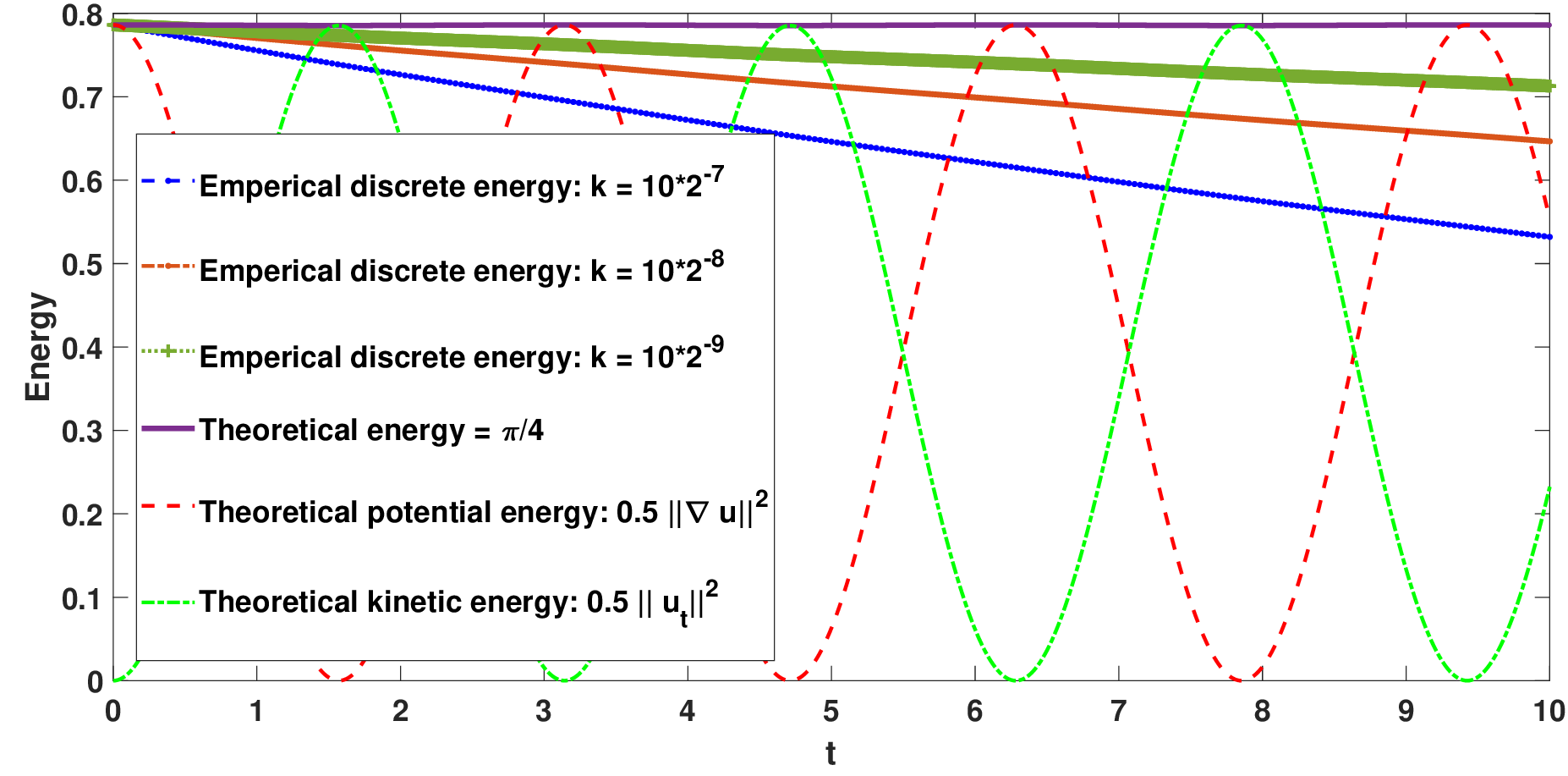}
 \includegraphics[width=6.1cm,height=3.2cm]
   {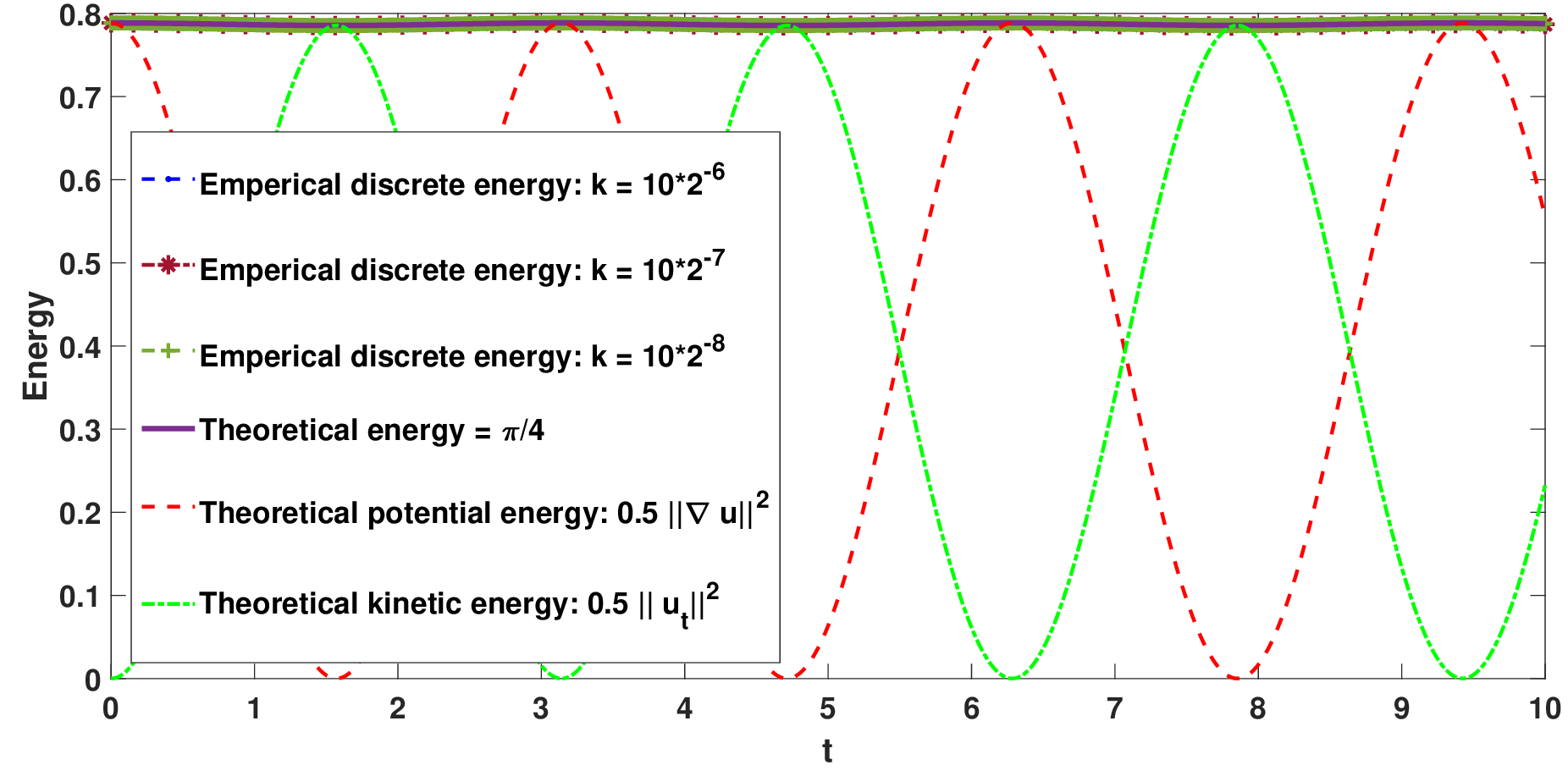}
 \caption{Comparison of the theoretical (continuous) energy 
 and the discrete energy for different values of the time step $k$: 
 (left) dG(0) (right) dG(1).}
\end{figure}
\begin{figure}[htp]  \label{Fig: one Discrete energy}
 \centering
 \includegraphics[width=6.cm,height=3.2cm]
   {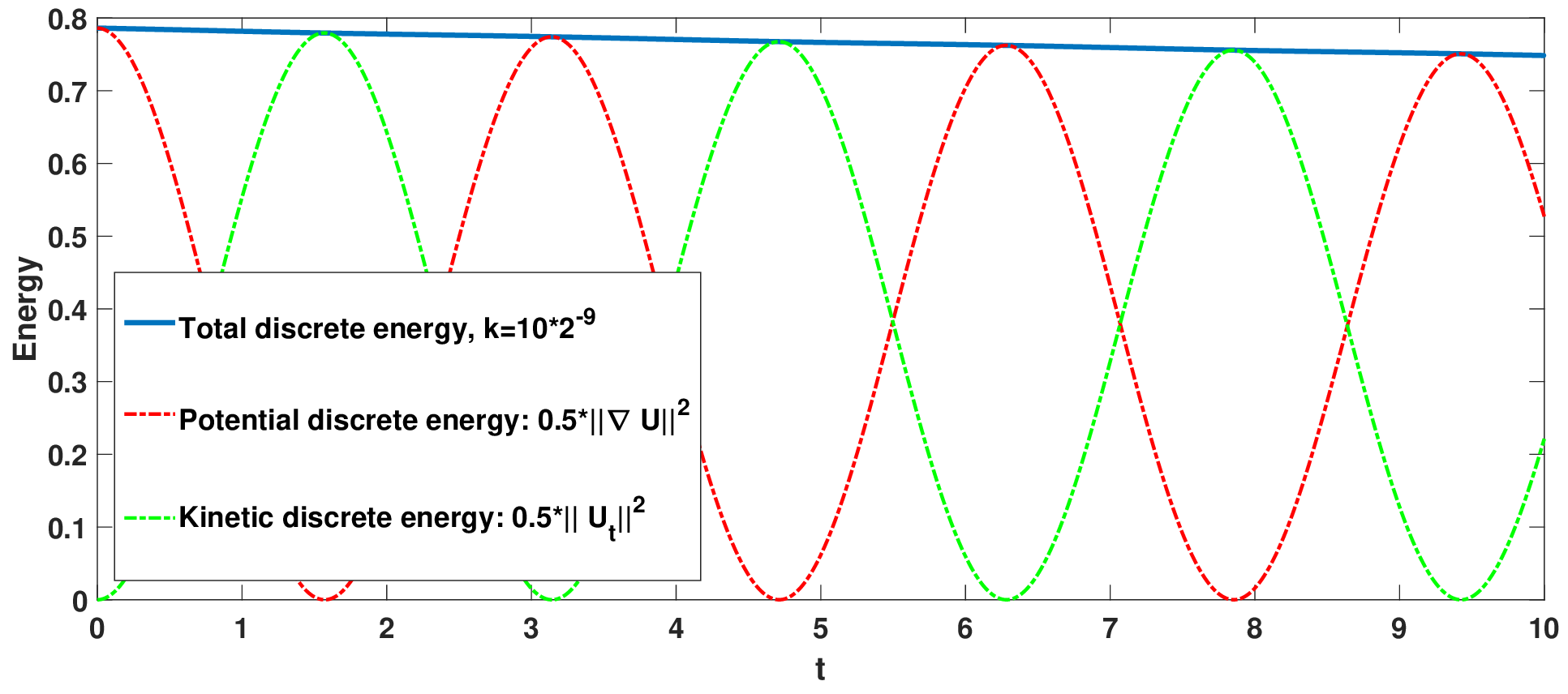}
 \includegraphics[width=6.cm,height=3.2cm]
   {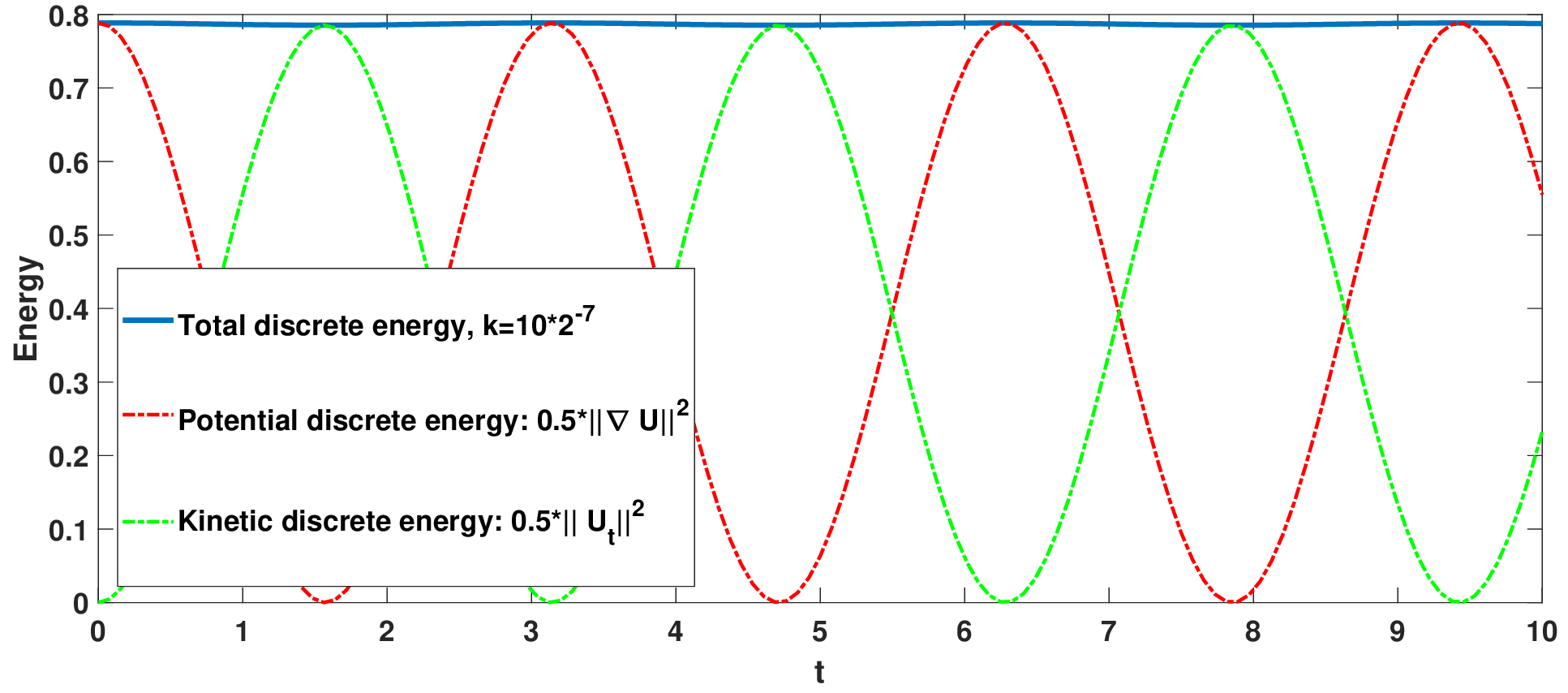}
 \caption{Discrete energy:  
 (left) dG(0) with $k=10\cdot 2^{-9}$ 
 (right) dG(1) with $k=10\cdot 2^{-7}$.}
\end{figure}

\newpage

\textbf{Acknowledgment.} 
We would like to thank the anonymous referee for constructive 
comments that helped us to improve the manuscript. 
We also thank Prof. Omar Lakkis for fruitful discussion 
and his constructive comments. 

\bibliographystyle{amsplain}

\end{document}